\documentclass[11pt]{amsart}
\usepackage{amsfonts,amssymb,amsthm,eucal,color,xypic}
\usepackage{mathrsfs,bbm,graphicx,hyperref}

\usepackage[margin=1.25in]{geometry}

\newtheorem{thm}{Theorem}
\newtheorem{cor}[thm]{Corollary}
\newtheorem{lemma}[thm]{Lemma}
\newtheorem{prop}[thm]{Proposition}
\newtheorem{conj}[thm]{Conjecture}

\theoremstyle{definition}

\newtheorem*{defn}{Definition}

\newcommand{\R}{\mathbb{R}}

\newcommand{\N}{\mathbb{N}}
\newcommand{\Z}{\mathbb{Z}}
\newcommand{\Q}{\mathbb{Q}}
\newcommand{\C}{\mathbb{C}}

\newcommand{\inprod}[2]{\left\langle #1, #2 \right\rangle}

\newcommand{\abs}[1]{\left\vert #1 \right\vert}
\newcommand{\norm}[1]{\left\Vert #1 \right\Vert}

\newcommand{\eps}{\varepsilon}
\newcommand{\dfn}[1]{\textbf{#1}}
\newcommand{\Set}[2]{\left\{#1 \mathrel{} \middle| \mathrel{} #2 \right\}}

\DeclareMathOperator{\vol}{vol}
\DeclareMathOperator{\supp}{supp}

\newcommand{\boxdim}{\dim_{\mathrm{Mink}}}

\newcommand{\mg}[1]{\left| #1 \right|}
\newcommand{\md}[1]{\mg{#1}_+}

\numberwithin{thm}{section}
\numberwithin{equation}{section}

\allowdisplaybreaks[3]

\newcommand{\scat}[1]{\mathbf{#1}}
\newcommand{\mr}[1]{\mathrm{#1}}

\newtheorem{examples}[thm]{Examples}

\newcommand{\from}{\colon}

\newcommand{\fcat}[1]{\mb{#1}}

\newcommand{\mb}[1]{\mathbf{#1}}

\newcommand{\transp}{\mathrm{T}}

\newcommand{\cat}[1]{\mathscr{#1}}
\newcommand{\Unit}{\mathbbm{1}}
\newcommand{\Vect}{\fcat{Vect}}
\DeclareMathOperator{\ob}{ob}		
\newcommand{\Hom}{\mr{Hom}}
\newcommand{\set}{\fcat{Set}}

\newcommand{\iso}{\cong}
\newcommand{\FinSet}{\fcat{FinSet}}
\newcommand{\FDVect}{\fcat{FDVect}}

\newcommand{\dee}{\,d}

\newcommand{\cube}[1]{\mathcal{C}^{#1}} 
\newcommand{\ball}[1]{\mathcal{B}^{#1}}

\author{Tom Leinster}

\address{School of Mathematics,
University of Edinburgh,
James Clerk Maxwell Building,
Peter Guthrie Tait Road,
Edinburgh EH9 3FD,
United Kingdom
}

\email{Tom.Leinster@ed.ac.uk}

\thanks{Tom Leinster was partially supported by an EPSRC Advanced Research Fellowship.}

\author{Mark W.\ Meckes}

\address{Department of Mathematics, Applied Mathematics, and
  Statistics, Case Western Reserve University, 10900 Euclid Ave.,
  Cleveland, Ohio 44106, U.S.A.}

\email{mark.meckes@case.edu}

\thanks{Mark Meckes was partially supported by a grant from the Simons Foundation (\#315593).}

\title[Magnitude: from categories to geometric measure theory]{The
  magnitude of a metric space: from category theory to geometric
  measure theory}

\keywords{magnitude,
metric space,
negative type,
Euler characteristic,
maximum diversity,
Minkowski dimension,
intrinsic volume}

\subjclass[2010]{Primary: 51F99;
Secondary:
49Q15,
28A75,
52A38,
31B15}

\begin{document}

\maketitle

\begin{abstract}
  Magnitude is a numerical isometric invariant of metric spaces, whose
  definition arises from a precise analogy between categories and
  metric spaces. Despite this exotic provenance, magnitude turns out
  to encode many invariants from integral geometry and geometric
  measure theory, including volume, capacity, dimension, and intrinsic
  volumes. This paper gives an overview of the theory of magnitude,
  from its category-theoretic genesis to its connections with these
  geometric quantities. Some new results are proved, including a
  geometric formula for the magnitude of a convex body in $\ell_1^n$.
\end{abstract}

\tableofcontents


\section{Introduction}
\label{S:intro}

Magnitude is a numerical isometric invariant of metric spaces.  Its
definition arises by viewing a metric space as a kind of
\emph{enriched category} --- an abstract structure that appears more
algebraic than geometric in nature --- and adapting a construction
from the intersection of category theory and homotopy theory.  One
would hardly expect, from such a provenance, that magnitude would have
any strong relationship to geometry as usually conceived.
Surprisingly, however, magnitude turns out to encode many invariants
from integral geometry and geometric measure theory, including volume,
capacity, dimension, and intrinsic volumes.  This paper will give an
overview of the theory of magnitude, from its category-theoretic
genesis to its connections with these geometric quantities.

We begin with a brief overview of the history of magnitude so far. The
grandparent of magnitude is the Euler characteristic of a topological
space, which is a natural analogue of the cardinality of a finite set.
To each category there is associated a topological space called its
\emph{classifying space}.  In \cite{ECC}, a formula was found for the
Euler characteristic of the classifying space of a suitably nice
finite category; applying this formula to less nice categories (for
which the Euler characteristic of the classifying space need not
exist) yielded a new cardinality-like invariant of categories, again
called the \emph{Euler characteristic} of a finite category.

Categories are a special case of a more general family of structures,
\emph{enriched categories}, which encompass both categories with
additional structure (like linear categories) and, surprisingly,
metric spaces.  In \cite{LeWi,MMS}, the definition of Euler
characteristic of a category was generalized to enriched categories,
renamed \emph{magnitude}, then re-specialized to finite metric spaces.
The first paper to be written on magnitude \cite{LeWi} focused on the
asymptotic behavior of the magnitudes of finite approximations to
specific compact subsets of Euclidean space.  The results there hinted
strongly that magnitude is closely related to geometric quantities
including volume and fractal dimension; numerical computations in
\cite{Willerton-heuristic} gave further evidence of these
relationships.

In \cite{Willerton-homog}, a definition was proposed for the magnitude
of certain compact metric spaces, and connections were found between
magnitude and some intrinsic volumes of Riemannian manifolds.  Shortly
thereafter, the paper \cite{MMS} appeared which laid out for the first
time the general theory of the magnitude of finite metric spaces; and
\cite{MeckPDM} which put the asymptotic approach of \cite{LeWi} for
studying magnitude of compact spaces on firm footing, and showed that
it also coincides with the definition used in \cite{Willerton-homog}.

The paper \cite{MeckMDC} introduced yet another equivalent approach to
magnitude for compact spaces, which makes magnitude more accessible to
a wide variety of analytic techniques. Using a result from potential
theory, \cite{MeckMDC} showed in particular that magnitude can be used
to recover the Minkowski dimension of a compact set in Euclidean
space.  Following the approach of \cite{MeckMDC}, the paper
\cite{BaCa} applied Fourier analysis to show that magnitude also
recovers volume in Euclidean space, and applied PDE techniques to
compute precisely magnitudes of Euclidean balls.

This paper aims to serve as a guide to the path from the definition of
the Euler characteristic of a finite category, to the geometric
results of \cite{MeckMDC} and \cite{BaCa} on magnitude in Euclidean
space.  It also includes a number of new results, in particular a
significant partial result toward a conjecture from \cite{MMS}
relating magnitude in $\ell_1^n$ to a family of intrinsic volumes
adapted to the $\ell_1$ metric, as well as generalizations of several
regularity results for magnitude from Euclidean space to more general
normed spaces. In order to reach the results of geometric interest as
quickly as possible, we omit many results from the papers named above,
and depart significantly at some points from the historical development
of ideas. We give complete proofs only for the new results, and for a
few known results for which we take a more direct approach than in
previous papers.

Section \ref{S:finite} begins with the definition of the Euler
characteristic of a finite category, and leads up to the magnitude of
a finite metric space and its basic properties.  Section
\ref{S:compact} covers the definition of the magnitude of a compact
space, its basic properties, and the results on magnitude of
manifolds. Section \ref{S:norms} covers magnitude in (quasi)normed
spaces, particularly $\ell_1^n$ and Euclidean space, and contains the
new results of this paper.  Finally, in section \ref{S:open-problems},
we discuss a number of open problems about magnitude.

Before moving on, we need to mention two threads in the story of
magnitude which have been ignored above and will make only brief
appearances in this paper.  The first is the magnitude of a graph,
viewed as a metric space with the shortest-path distance between
vertices.  This subject has been developed in \cite{MG}, which in
particular investigated its relationship to classical, combinatorial
graph invariants, and \cite{HeWi}, which found that the magnitude of
graphs is the Euler characteristic associated to a graded homology
theory for graphs.  The second thread is the connection of magnitude
to quantifying biodiversity and maximum entropy problems. This is
actually related with the historically first appearance of the
magnitude of a metric space in the literature, in \cite{SoPo}, and was
developed in \cite{METAMB,LeMe}; section \ref{S:diversity} will take
half a step in the direction of these connections.

\section{Finite metric spaces}
\label{S:finite}

Here we explain the origins of the notion of magnitude.  There is a simple
combinatorial definition of the magnitude or Euler characteristic of a
finite category (section~\ref{S:EC}), which extends in a natural way to a
more general class of structures, the enriched categories
(section~\ref{S:enriched}).  As we show, this general invariant is closely
related to several existing invariants of size.  Specializing it in a
different direction gives the definition of the magnitude of a finite
metric space (sections~\ref{S:finite-magnitude} and~\ref{S:pd}).

In order to do any of this, we first need to define the magnitude of a
matrix.

\subsection{The magnitude of a matrix}
\label{S:matrices}

Recall that a \dfn{semiring} is a ``ring without negatives'', that is,
an abelian group (written additively) with an associative operation of
multiplication that distributes over addition. Let $k$ be a
commutative semiring (always assumed to have a multiplicative identity
$1$) and $A$ a finite set, and let $Z \in k^{A \times A}$ be a square
matrix over $k$ indexed by the elements of $A$.  A \dfn{weighting} on
$Z$ is a column vector $w \in k^A$ satisfying $Zw = e$, where $e$ is
the column vector of $1$s, and a \dfn{coweighting} on $Z$ is a row
vector $v \in k^A$ satisfying $vZ = e^\transp$. That is,
\[
\sum_{b \in A} Z(a,b) w_b = 1 \text{ for every } a \in A
\]
and 
\[
\sum_{a \in A} v_a Z(a,b) = 1 \text{ for every } b \in A.
\]  
If $w$ is a weighting and $v$ a coweighting on $Z$ then 
\[
\sum_{a \in A} w_a = e^\transp w = vZw = v e = \sum_{a \in A} v_a.
\]
When $Z$ admits both a weighting and a coweighting, we may therefore
define the \dfn{magnitude} $\mg{Z}$ of $Z$ to be the common quantity
$\sum_a w_a = \sum_a v_a$, for any weighting $w$ and coweighting $v$.

An important special case is when $Z$ is invertible.  Then $Z$ has a unique
weighting and a unique coweighting, and its magnitude is the sum of the
entries of $Z^{-1}$:
\begin{equation}
  \label{E:inverse-sum}
  \mg{Z} = \sum_{a, b \in A} Z^{-1}(a, b).
\end{equation}
An even more special case is that of positive definite matrices:

\begin{prop}
  \label{T:pd-sup}
  Let $Z \in \R^A$ be a positive definite matrix.  Then
  \[
  \mg{Z} = \sup_{0 \neq x \in \R^A} \frac{(\sum_a x_a)^2}{x^\transp Z x},
  \]
  and the supremum is attained exactly when $x$ is a scalar multiple of the
  unique weighting on $Z$.
\end{prop}

This follows swiftly from the Cauchy--Schwarz inequality
\cite[Proposition~2.4.3]{MMS}.

\subsection{The Euler characteristic of a finite category}
\label{S:EC}

A category can be viewed as a directed graph (allowing multiple parallel
edges) together with an associative, unital operation of composition.  The
vertices of the graph are the objects of the category, and for each pair
$(a, b)$ of vertices, the edges from $a$ to $b$ in the graph are the maps
from $a$ to $b$ in the category, which form a set $\Hom(a, b)$.  Thus,
composition defines a function $\Hom(a, b) \times \Hom(b, c) \to \Hom(a,
c)$ for each $a, b, c$, and there is a loop $1_a \in \Hom(a, a)$ on each
vertex $a$.  Although in many categories of interest, the collections of
objects and maps form infinite sets or even proper classes, we will be
considering \dfn{finite categories}: those with only finitely many objects
and maps.

Let $\mathbf{A}$ be a finite category, with set of objects $\ob
\mathbf{A}$.  The \dfn{Euler characteristic} of $\mathbf{A}$ is the
magnitude of the matrix $Z_\mathbf{A} \in \Q^{\ob \mathbf{A} \times \ob
  \mathbf{A}}$ given by $Z_\mathbf{A}(a,b) = \# \Hom(a,b)$ (where $\#$
denotes cardinality), whenever this magnitude is defined.

For example, if $\mathbf{A}$ has no maps other than identities then
$Z_\mathbf{A}$ is the identity and the Euler characteristic of
$\mathbf{A}$ is simply the number of objects.  More generally, any
partially ordered set $(P, \leq)$ gives rise to a category
$\mathbf{A}$ whose objects are the elements of $P$, and with one map
$a \to b$ when $a \leq b$ and none otherwise.  In a theory made famous
by Rota \cite{RotaFCT}, every finite partially ordered set $P$ has
associated with it a M\"obius function $\mu$, which is defined on
pairs $(a, b)$ of elements of $P$ such that $a \leq b$, and takes
values in $\Z$.  It generalizes the classical M\"obius function, and
the construction above for categories generalizes it further still:
$\mu(a, b) = Z_\mathbf{A}^{-1}(a, b)$ whenever $a \leq b$, and the
definition of Euler characteristic of a category extends the existing
definition for ordered sets \cite[Proposition~4.5]{ECC}.

To any small category $\mathbf{A}$ there is assigned a topological
space, called its \dfn{classifying space}.  The name ``Euler
characteristic'' is largely justified by the following result.

\begin{thm}[{\cite[Proposition 2.11]{ECC}}]
  \label{T:EC=EC}
  Let $\mathbf{A}$ be a finite category.  Under appropriate conditions
  (which imply, in particular, that the Euler characteristic of the
  classifying space of $\mathbf{A}$ is defined), the Euler
  characteristic of the category $\mathbf{A}$ is equal to the Euler
  characteristic of its classifying space.
\end{thm}

Euler characteristic for finite categories enjoys many properties
analogous to those enjoyed by topological Euler characteristic
\cite[Section~2]{ECC}.  For instance, categorical Euler characteristic is
invariant under equivalence (mirroring homotopy invariance in the
topological setting), and is additive with respect to disjoint union of
categories and multiplicative with respect to products.  There is even an
analogue of the topological formula for the Euler characteristic of the
total space of a fibration.

Schanuel \cite{SchaNSE} argued that Euler characteristic for
topological spaces is closely analogous to cardinality for sets.  For
instance, it has analogous additivity and multiplicativity properties,
it satisfies the inclusion-exclusion principle (under hypotheses),
and, indeed, it reduces to cardinality for finite discrete spaces.
Similarly, the results described above suggest that Euler
characteristic for finite categories is the categorical analogue of
cardinality.

\subsection{Enriched categories}
\label{S:enriched}

A \dfn{monoidal category} is a category $\cat{V}$ equipped with an
associative binary operation $\otimes$ (which is formally a functor $\cat{V}
\times \cat{V} \to \cat{V}$) and a unit object $\Unit \in \cat{V}$.  The
associativity and unit axioms are only required to hold up to suitably
coherent isomorphism; see~\cite{MacLCWM} for details.

Typical examples of monoidal categories $(\cat{V}, \otimes, \Unit)$ are the
categories $(\set, \times, \{\star\})$ of sets with cartesian product and
$(\FDVect_K, \otimes, K)$ of finite-dimensional vector spaces over a field
$K$.  A less obvious example is the ordered set $([0, \infty], \geq)$.  As
a category, its objects are the nonnegative reals together with $\infty$,
there is one map $x \to y$ when $x \geq y$, and there are none otherwise.
It is monoidal with $\otimes = +$ and $\Unit = 0$.

Let $\cat{V} = (\cat{V}, \otimes, \Unit)$ be a monoidal category.  The
definition of category enriched in $\cat{V}$, or $\cat{V}$-category, is
obtained from the definition of ordinary category by requiring that the
hom-sets are no longer sets but objects of $\cat{V}$.  Thus, a (small)
\dfn{$\cat{V}$-category} $\scat{A}$ consists of a set $\ob\scat{A}$ of
objects, an object $\Hom(a, b)$ of $\cat{V}$ for each $a, b \in
\ob\scat{A}$, and operations of composition and identity satisfying
appropriate axioms~\cite{KellBCE}.  The composition consists of a map 
\[
\Hom(a, b) \otimes \Hom(b, c) \to \Hom(a, c)
\]
in $\cat{V}$ for each $a, b, c \in \ob\scat{A}$, while the identities are
provided by a map $\Unit \to \Hom(a, a)$ for each $a \in \ob\scat{A}$.  

\begin{examples}
\begin{enumerate}
\item 
When $\cat{V} = \set$ (with monoidal structure as above), a
$\cat{V}$-category is an ordinary (small) category.

\item
When $\cat{V} = \Vect_K$, a $\cat{V}$-category is a \dfn{linear
  category}, that is, a category in which each hom-set carries the
structure of a vector space, and composition is bilinear.

\item
When $\cat{V} = [0, \infty]$, a $\cat{V}$-category is a \dfn{generalized
  metric space} \cite{LawvMSG,LawvTCS}.  That is, a $\cat{V}$-category
consists of a set $A$ of objects or points together with, for each $a, b
\in A$, a real number $\Hom(a, b) = d(a, b) \in [0, \infty]$, satisfying
the axioms
\[
d(a, b) + d(b, c) \geq d(a, c),
\qquad
d(a, a) = 0
\]
($a, b, c \in A$).  Such spaces are more general than classical metric
spaces in three ways: $\infty$ is permitted as a distance, the
separation axiom $d(a, b) = 0 \implies a = b$ is dropped, and, most
significantly, $d$ is not required to be symmetric.

\item
The category $\cat{V} = ([0,\infty],\geq)$ can alternatively be given
the monoidal structure $(\max, 0)$. A $\cat{V}$-category is then a
generalized ultrametric space, that is, a generalized metric space
satisfying the stronger triangle inequality $\max\{d(a,b), d(b,c)\}
\ge d(a,c)$.
\end{enumerate}
\end{examples}

To define the magnitude of an enriched category, we start with a monoidal
category $(\cat{V}, \otimes, \Unit)$ together with a commutative semiring
$k$ and a map $\mg{\,\cdot\,}\from \ob\cat{V} \to k$, with the property
that $\mg{X} = \mg{Y}$ whenever $X \iso Y$, and satisfying the
multiplicativity axioms $\mg{X \otimes Y} = \mg{X} \cdot \mg{Y}$ and
$\mg{\Unit} = 1$.

\begin{defn}
Let $\scat{A}$ be a $\cat{V}$-category with only finitely many objects.
\begin{enumerate}
\item
The \dfn{similarity matrix} of $\scat{A}$ is the $\ob\scat{A} \times
\ob\scat{A}$ matrix $Z_\scat{A}$ over $k$ defined by $Z_\scat{A}(a, b) =
\mg{\Hom(a, b)}$.

\item
A \dfn{(co)weighting} on $\scat{A}$ is a (co)weighting on $Z_\scat{A}$,
and $\scat{A}$ \dfn{has magnitude} if $Z_\scat{A}$ does.  Its
\dfn{magnitude} is then $\mg{\scat{A}} = \mg{Z_\scat{A}}$.
\end{enumerate}
\end{defn}

\begin{examples}
\begin{enumerate}
\item 
Let $\cat{V}$ be the monoidal category $(\FinSet, \times, \{\star\})$ of
finite sets.  Let $k = \Q$, and for $X \in \FinSet$, let $\mg{X} \in \Q$ be
the cardinality of $X$.  Then we obtain a notion of magnitude for finite
categories; it is exactly the Euler characteristic of
section~\ref{S:EC}.

\item Let $\cat{V}$ be the monoidal category $\FDVect_K$ of
  finite-dimensional vector spaces over a field $K$.  Let $k = \Q$, and for
  $X \in \FDVect_K$, put $\mg{X} = \dim X \in \Q$.  Then we obtain a notion
  of magnitude for linear categories with finitely many objects and
  finite-dimensional hom-spaces. As shown in \cite{MFDA}, this invariant is
  closely related to the Euler form of an associative algebra, defined
  homologically.

\item Let $\cat{V} = [0, \infty]$, with monoidal structure $(+,
  0)$. Let $k = \R$, and for $x \in [0, \infty]$, put $\mg{x} =
  e^{-x}$.  (We have little choice about this: the multiplicativity
  axioms force $\mg{x} = C^x$ for some constant $C$, at least assuming
  that $\mg{\cdot}$ is to be measurable. We will address the one
  degree of freedom here through the introduction of magnitude
  functions in the next section.)  Then we obtain a notion of the
  magnitude $\mg{A} \in \R$ of a finite metric space $\mg{A}$, examined in
  detail later.

\item Let $\cat{V}=[0, \infty]$, now with monoidal structure $(\max,
  0)$. Let $k = \R$, and define $\mg{\cdot}: [0, \infty] \to \R$ to be
  either the indicator function of $[0, 1]$ or that of $[0, 1)$.  It
  is shown in Section 8 of \cite{MeckMDC} that these are essentially
  the only possibilities for $\mg{\cdot}$, and that the resulting
  magnitude of a finite ultrametric space is simply the number of
  balls of radius $1$ (closed or open, respectively) needed to cover
  it.  It is also shown that this leads naturally to the notion of
  $\eps$-entropy or $\eps$-capacity.

\end{enumerate}
\end{examples}

The multiplicativity condition $\mg{X \otimes Y} = \mg{X} \cdot
\mg{Y}$ on objects of $\cat{V}$ has so far not been used.  However, it
implies a similar multiplicativity condition on categories enriched in
$\cat{V}$.  In the case of metric spaces, this reduces to Proposition
\ref{T:finite-ell1-product} below; for the general statement, see
\cite[Proposition 1.4.3]{MMS}.

\subsection{The magnitude of a finite metric space}
\label{S:finite-magnitude}

Concretely, the magnitude $\mg{A}$ of a finite metric space $(A, d)$
is the magnitude of the matrix $Z = Z_A \in \R^{A \times A}$ given by
$Z_A(a,b) = e^{-d(a,b)}$, if that is defined.  Taking advantage of the
symmetry of $Z_A$ to simplify slightly, this means the following.  A vector
$w \in \R^A$ is a \dfn{weighting} for $A$ if $Z_A w = e$, where $e \in
\R^A$ is the column vector of $1$s, and if a weighting for $A$ exists, then
the \dfn{magnitude} of $A$ is
\[
\mg{A} = \sum_{a \in A} w_a.
\]

This is not a classical invariant or one that appears to have previously
been explored mathematically prior to the work cited in the introduction.
Neither is it wholly new.  In a probabilistic analysis of the benefits of
highly diverse ecosystems, Solow and Polasky \cite{SoPo} derived a
lower bound on the benefit and identified one term, which they called the
``effective number of species'', as especially interesting.  Although it
was not thoroughly investigated in \cite{SoPo}, this term is exactly our
magnitude.  The reader is referred to \cite{MMS,LeCo,METAMB,LeMe} for more
information about this connection.

Not every finite metric space possesses a weighting or, therefore, has
well-defined magnitude. One large and important class of spaces which
always does is the subject of section \ref{S:pd}.  The next two results
give additional examples.

From now on, to simplify the statements of results, all metric spaces and
all compact sets in a metric space are assumed to be nonempty.

\begin{prop}[{\cite[Theorem 2]{LeWi} and \cite[Proposition
    2.1.3]{MMS}}] 
  \label{T:scattered}
  Let $(A,d)$ be a finite metric space, and suppose that whenever $a,
  b \in A$ with $a \neq b$, we have $d(a,b) > \log(\# A - 1)$. Then
  $A$ possesses a positive weighting, and $\mg{A}$ is therefore
  defined.
\end{prop}

A metric space $(A,d)$ is called \dfn{homogeneous} if its isometry
group acts transitively on the points of $A$.

\begin{prop}[{\cite{SpeyMS}; see also \cite[Proposition 2.1.5]{MMS}}]
  \label{T:Speyer}
  If $(A,d)$ is a finite homogeneous metric space and $a_0 \in A$ is
  any fixed point, then $A$ possesses a positive weighting and
  \[
  \mg{A} = \frac{(\# A)^2}{\sum_{a,b \in A} e^{-d(a,b)}} = \frac{\#
    A}{\sum_{a \in A} e^{-d(a,a_0)}}.
  \]
\end{prop}

For metric spaces $(A,d_A)$ and $(B,d_B)$, we denote by $A \times_1 B$
the set $A \times B$ equipped with the metric
\[
d\bigl((a, b), (a', b')\bigr) = d_A(a,a') + d_B(b,b').
\]

\begin{prop}[{\cite[Proposition 2.3.6]{MMS}}]
  \label{T:finite-ell1-product} 
  Suppose that $(A,d_A)$ and $(B,d_B)$ are finite metric spaces with
  weightings $w \in \R^A$ and $v \in \R^B$ respectively. Then $x \in
  \R^{A\times B}$ given by $x_{(a,b)} = w_a v_b$ is a weighting for $A
  \times_1 B$, and $\mg{A \times_1 B} = \mg{A} \mg{B}$.
\end{prop}

Proposition \ref{T:finite-ell1-product} has a generalization, Theorem
2.3.11 of \cite{MMS}, which is an analogue for magnitude of the
formula for the Euler characteristic of the total space of a
fibration.

As noted earlier, there is an arbitrary choice of scale implicit in
the definition of magnitude: we could choose any other base for the
exponent in place of $e^{-1}$. To deal with this, we will often work
with the whole family of metric spaces $\{tA\}_{t > 0}$, where $tA$
denotes the metric space $(A,td)$.  We will sometimes also let $0A$
denote a one-point space.  The (partially defined) function $t \mapsto
\mg{tA}$ is called the \dfn{magnitude function} of $A$.

\begin{prop}[{\cite[Proposition 2.2.6]{MMS}}]
  \label{T:large-t}
  Let $(A,d)$ be a finite metric space.
  \begin{enumerate}
  \item \label{P:large-t-defined}
    $\mg{tA}$ is defined for all but finitely many $t > 0$.
    
  \item \label{P:large-t-inc} 
    For sufficiently large $t$, $\mg{tA}$ is an increasing function of $t$.

  \item \label{P:large-t-limit} 
    $\lim_{t \to \infty} \mg{tA} = \#A$.
  \end{enumerate}
\end{prop}

Proposition \ref{T:large-t} supports the interpretation of the
magnitude $\mg{tA}$ as the ``effective number of points'' in $A$, when
viewed as a scale determined by $t$.  (We recall Solow and Polasky's
interpretation of $\mg{A}$ as the ``effective number of species''.)
However, the hypotheses of the propositions above also highlight the
counterintuitive behaviors that magnitude may exhibit.  In particular,
there exists a metric space $A$ such that each of the following holds:
\begin{enumerate}
\item $\mg{tA}$ is undefined for some $t > 0$.
\item $\mg{tA}$ is decreasing for some $t > 0$.
\item $\mg{tA} < 0$ for some $t > 0$.
\item There exists a $B \subseteq A$ such that $\mg{tB} > \mg{tA}$ for
  some $t > 0$.
\end{enumerate}
We need not look that hard to find such an ill-behaved space: the
complete bipartite graph $K_{3,2}$, equipped with the shortest path
metric, has all these unpleasant properties; see Example 2.2.7 of
\cite{MMS}. In the next section we will consider a class of spaces
which avoids most of these pathologies.

We end this section by noting that the issue of scale can be dealt
with in a more elegant way if $A$ is the vertex set of a graph and $d$
is the shortest path metric, or more generally, whenever $d$ is
integer-valued. By \eqref{E:inverse-sum}, in this situation $\mg{tA}$
is a rational function of $q = e^{-t}$.  More directly, if one
restricts attention to such spaces, the semiring $k$ in the previous
section can be taken to be the ring $\Q(q)$ of rational functions in a
formal variable $q$.  Then the matrix $Z_A \in (\Q(q))^{A\times A}$ is
always invertible, so the magnitude $\mg{A}$ is always defined as an
element of $\Q(q)$; see section 2 of \cite{MG}.

\subsection{Positive definite metric spaces}
\label{S:pd}

As noted in section \ref{S:matrices}, a positive definite matrix $Z$
always has magnitude, given by Proposition \ref{T:pd-sup}.  We will
now explore the consequences of this observation for magnitude of
metric spaces.

A finite metric space $(A,d)$ is said to be \dfn{positive definite} if
the associated matrix $Z_A$ is positive definite, and is said to be
of \dfn{negative type} if $Z_{tA}$ is positive semidefinite
for every $t > 0$. It can be shown \cite[Theorem 3.3]{MeckPDM} that if
$(A,d)$ is of negative type, then in fact $Z_{tA}$ is positive
definite, and hence $tA$ is a positive definite space. A general
metric space is said to be positive definite or of negative type,
respectively, if every finite subspace is.


The strange turn of terminology here is due to the negative sign in
$e^{-d}$. Negative type has several other equivalent formulations, and
is an important property in the theory of metric embeddings (see,
e.g., \cite{DeLa,BeLi,WeWi}).  The fact that negative type appears
naturally when considering magnitude is a hint that magnitude does in
fact connect with more classical topics in geometry.

The following result is an immediate consequence of Proposition
\ref{T:pd-sup} and the definition of magnitude.

\begin{prop}[{\cite[Proposition 2.4.3]{MMS}}]
  \label{T:finite-sup}
  If $A$ is a finite positive definite metric space, then the
  magnitude $\mg{A}$ is defined, and
  \[
  \mg{A} = \max_{0 \neq x \in \R^A} \frac{\left(\sum_{a\in A}
      x_a\right)^2}{x^\transp Z_A x},
  \]
  and the supremum is attained exactly when $x$ is a scalar multiple of the
  unique weighting on $A$.
\end{prop}

A first application of Proposition \ref{T:finite-sup} is Proposition
\ref{T:scattered}, which is proved by showing that for large enough
$t$, $Z_{tA}$ is positive definite.

\begin{cor}[Corollaries 2.4.4 and 2.4.5 of \cite{MMS}]
  \label{T:finite-monotone}
  If $A$ is a finite positive definite metric space and $\emptyset
  \neq B \subseteq A$, then $1 \le \mg{B} \le \mg{A}$.
\end{cor}

Proposition \ref{T:finite-sup} will also be one of our main tools in
the extension of magnitude to compact spaces in section
\ref{S:compact}.

Proposition \ref{T:finite-sup} and its consequences would be of little
interest without a large supply of interesting examples of positive
definite spaces.  Many are collected in the following result; we refer
to \cite[Theorem 3.6]{MeckPDM} for references and further examples.

\begin{thm}
  \label{T:pd-examples}
  The following metric spaces are of negative type, and thus magnitude
  is defined for all their finite subsets.
  \begin{enumerate}
  \item 
    $\ell_p^n$, the set $\R^n$ equipped with the metric derived from the
    $\ell_p$-norm, for $n \geq 1$ and $1 \leq p \leq 2$;
    
  \item
    Lebesgue space $L_p[0, 1]$, for $1 \leq p \leq 2$;
  \item
    round spheres (with the geodesic distance);
    
  \item
    real and complex hyperbolic space;
    
  \item
    ultrametric spaces;

  \item weighted trees.
  \end{enumerate}
\end{thm}
 
Furthermore, some natural operations on positive definite spaces yield new
positive definite spaces.

\begin{prop}[{\cite[Lemma 2.4.2]{MMS}}]
  \label{T:finite-pd-operations} ~
  \begin{enumerate}
  \item Every subspace of a positive definite metric space is positive
    definite.
  \item If $A$ and $B$ are positive definite metric spaces, then $A
    \times_1 B$ is positive definite.
  \end{enumerate}
\end{prop}

On the other hand, many spaces of geometric interest are \emph{not} of
negative type, and many natural operations fail to preserve positive
definiteness; see \cite[Section 3.2]{MeckPDM} for examples and
references.

\section{Compact metric spaces}
\label{S:compact}

Despite strong and growing interest in the geometry of finite metric
spaces (see e.g.\ \cite{Linial}), it is natural to try to define an
invariant of metric spaces, like magnitude, more generally. The most
obvious context is that of compact spaces.  The general definition of
the magnitude of an enriched category does not help us here, but
several strategies present themselves, including approximating a
compact space by finite subspaces and generalizing the notion of a
weighting to compact spaces.  In section \ref{S:compact-pd} we will
see that there is a canonical (hence ``correct'') extension of
magnitude from finite metric spaces to compact positive definite
spaces, which can be formulated in several ways.  In section
\ref{S:weight-measures} we will investigate a generalization of
weightings to compact spaces, and see that this approach to defining
magnitude agrees with the former one. This approach is of more limited
scope, but often gives the easiest approach to computing magnitude;
using it, we will see that magnitude knows about at least some
intrinsic volumes of certain Riemannian manifolds.  Finally, section
\ref{S:diversity} will introduce another invariant, maximum diversity,
which is closely related to magnitude, and will be a crucial tool in
proving the connection between magnitude and Minkowski dimension.

\subsection{Compact positive definite spaces}
\label{S:compact-pd}

To justify the ``correctness'' of our definition of magnitude for
compact positive definite spaces, we need a topology on the family of
(isometry classes of) compact metric spaces.  Recall that the
\dfn{Hausdorff metric} $d_H$ on the family of compact subsets of a
metric space $X$ is given by
\[
d_H(A,B) = \max\bigl\{ \sup_{a \in A} d(a,B), \sup_{b \in B} d(b,A)\bigr\}.
\]
The \dfn{Gromov--Hausdorff distance} between two compact metric spaces
$A$ and $B$ is
\[
d_{GH}(A,B) = \inf d_H\bigl(\varphi(A),\psi(B)\bigr),
\]
where the infimum is over all metric spaces $X$ and isometric
embeddings $\varphi:A \to X$ and $\psi:B \to X$.  This defines a
metric on the family of isometry classes of compact metric spaces; see
\cite[Chapter 3]{GromMSR}.

The following result follows from the proof of \cite[Theorem
2.6]{MeckPDM}, although our definitions are organized rather
differently in that paper. We give a more streamlined version of the
argument from \cite{MeckPDM}.

\begin{prop}
  \label{T:semicont}
  The quantity
  \begin{equation}
    \label{E:sup}
    M(A) = \sup \Set{\mg{A'} }{A' \subseteq A, \ A' \text{ finite}}
  \end{equation}
  is lower semicontinuous as a function of $A$ (taking values in
  $[0,\infty]$), on the class of compact positive definite metric
  spaces equipped with the Gromov--Hausdorff topology.
\end{prop}

\begin{proof}
  Suppose first that $d_{GH}(A,B) < \delta$ for \emph{finite} positive
  definite spaces $A$ and $B$, and let $w \in \R^A$ be a weighting for
  $A$.  There is a function $f:A \to B$ such that
  $\abs{d(f(a),f(a')) - d(a,a')} < 2\delta$ for all $a,a' \in A$. Define
  $v \in \R^B$ by $v_b = \sum_{a \in f^{-1}(b)} w_a$, and
  $Z_f \in \R^{A \times A}$ by $Z_f(a,a') = e^{-d(f(a),f(a'))}$. Then
  $v^\transp Z_B v = w^\transp Z_f w$, and so
  \[
  \abs{w^\transp Z_A w - v^\transp Z_B v}
  = \abs{ w^\transp (Z_A - Z_f) w} \le \norm{w}_1^2 \norm{Z_A -
    Z_f}_\infty
  < 2 \norm{w}_1^2 \delta.
  \]
  Thus by Proposition \ref{T:finite-sup},
  \begin{equation}
    \label{E:semicont-finite-bound}
  \mg{B} \ge \frac{(\sum_b v_b)^2}{v^\transp Z_B v} \ge \frac{(\sum_a
    w_a)^2}{w^\transp Z_A w + 2 \norm{w}_1^2 \delta} =
  \frac{\mg{A}^2}{\mg{A} + 2 \norm{w}_1^2 \delta} \ge \mg{A} - 2
  \norm{w}_1^2 \delta.
  \end{equation}

  Now for general $A$, assume for simplicity that $M(A) < \infty$ (the
  case $M(A) = \infty$ is handled similarly). Given $\eps > 0$, pick a
  finite subset $A' \subset A$ such that $\mg{A'} \ge M(A) - \eps$,
  and let $w \in \R^{A'}$ be a weighting for $A'$. If $d_{GH}(A,B) <
  \delta$, then there is a finite subset $B'\subseteq B$ such that
  $d_{GH}(A',B') < \delta$, and so by \eqref{E:semicont-finite-bound},
  \[
  M(B) \ge \mg{B'} \ge \mg{A'} - 2 \norm{w}_1^2 \delta \ge M(A) -
  \eps - 2 \norm{w}_1^2 \delta.
  \]
  Therefore $M(B) \ge M(A) - 2 \eps$ when $d_{GH}(A,B)$ is
  sufficiently small.
\end{proof}

Corollary \ref{T:finite-monotone} implies that $M(A) = \abs{A}$ when
$A$ itself is finite and positive definite.  Proposition
\ref{T:semicont} thus implies first of all that magnitude is l.s.c.\
on the class of finite positive definite metric spaces.  It follows
that there is a canonical extension of magnitude to the class of
compact positive definite metric spaces, namely, the \emph{maximal}
l.s.c.\ extension.  Proposition \ref{T:semicont} furthermore implies
that this extension is precisely the function $M$ in \eqref{E:sup}.
For a compact positive definite metric space $(A,d)$, we therefore
define the \dfn{magnitude} $\mg{A}$ to be the value of the supremum
$M(A)$ in \eqref{E:sup}.

Thus magnitude is lower semicontinuous on the class of compact
positive spaces.  This cannot be improved to continuity in general,
even for the class of finite spaces of negative type.  Examples 2.2.8
and 2.4.9 in \cite{MMS} discuss a space $A$ of negative type with six
points, such that $\mg{tA} = 6/(1+4e^{-t})$; thus
$\lim_{t\to 0^+} \mg{tA} = 6/5$, whereas the space $tA$ itself
converges to a one-point space.  On the other hand, magnitude is
continuous when restricted to certain classes of spaces, as we will
see in Corollary \ref{T:pw-mag-cont} and Theorem
\ref{T:star-continuity} below.

Proposition \ref{T:semicont}, Proposition \ref{T:finite-ell1-product},
and Corollary \ref{T:finite-monotone}
yield the following results.

\begin{prop}[{\cite[Lemma 3.1.3]{MMS}}]
  \label{T:compact-monotone}
  If $A$ is a compact positive definite metric space and $\emptyset
  \neq B \subseteq A$, then $1 \le \mg{B} \le \mg{A}$.
\end{prop}

\begin{prop}[{\cite[Corollary 2.7]{MeckPDM}}]
  \label{T:finite-approximation}
  Let $A$ be a compact positive definite metric space, and let
  $\{A_k\}$ be any sequence of compact subsets of $A$ such that $A_k
  \xrightarrow{k \to \infty} A$ in the Hausdorff topology. Then
  $\mg{A} = \lim_{k\to \infty} \mg{A_k}$.
\end{prop}

\begin{prop}[{\cite[Proposition 3.1.4]{MMS}}]
  \label{T:compact-ell1-product}
  If $A$ and $B$ are compact positive definite metric spaces, then
  $\mg{A \times_1 B} = \mg{A} \mg{B}$.
\end{prop}

Proposition \ref{T:semicont} justifies the above definition of
magnitude as the ``correct'' one for a compact positive definite space
$A$. Nevertheless, for both {\ae}sthetic and practical reasons, it is
desirable to be able to work directly with $A$ itself, as opposed to
approximations of $A$ by finite subspaces. Two different more direct
approaches to defining magnitude for compact positive definite spaces
were developed in \cite{MeckPDM,MeckMDC}. In essence, these papers
introduced two different topologies on the space
$\Set{ w \in \R^A }{\supp w \text{ is finite}}$. The topology used in
\cite{MeckPDM} has the advantage of being more familiar, whereas the
topology in \cite{MeckMDC} has the advantage of being better suited to
the analysis of magnitude. In particular, the topology used in
\cite{MeckMDC} can be dualized in a way that presents a new set of
tools to study magnitude.  In the pursuit of our goal of proceeding as
quickly as possible to geometric results, here we will go straight to
the dual version.

Recall that a \dfn{positive definite kernel} on a space $X$ is a
function $K: X \times X \to \C$ such that, for every finite set
$A \subseteq X$, the matrix $[K(a,b)]_{a,b \in A} \in \C^{A \times A}$
is positive definite.  Given a positive definite kernel on $X$, the
\dfn{reproducing kernel Hilbert space} (RKHS) $\mathcal{H}$ on $X$
with kernel $K$ is the completion of the linear span of the functions
$k_x(y) = K(x,y)$ with respect to the inner product given by
\[
\inprod{k_x}{k_y}_{\mathcal{H}} = K(x,y)
\]
(see \cite{Aronszajn}).  If $f \in \mathcal{H}$, then $f(x) =
\inprod{f}{k_x}_{\mathcal{H}}$ for every $x \in X$, and consequently
\begin{equation}
  \label{E:RKHS-bound}
  \abs{f(x)} \le \norm{f}_{\mathcal{H}} \norm{k_x}_{\mathcal{H}}
  = \norm{f}_{\mathcal{H}} \sqrt{K(x,x)}
\end{equation}
by the Cauchy--Schwarz inequality.

Now if $(X,d)$ is a positive definite metric space, then
$K(x,y) = e^{-d(x,y)}$ is a positive definite kernel on $X$. We will
refer to the corresponding RKHS as the RKHS $\mathcal{H}$ for $X$.

\begin{thm}[{\cite[Theorem 4.1 and Proposition 4.2]{MeckMDC}}]
  \label{T:inf}
  Let $X$ be a positive definite metric space, and let $A \subseteq X$
  be compact. Then $\mg{A} < \infty$ if and only if there exists a
  function $h \in \mathcal{H}$ such that $h \equiv 1$ on $A$.  In that
  case,
  \[
  \mg{A} = \inf \Set{ \norm{h}_\mathcal{H}^2 }{h \in \mathcal{H}, \ h
    \equiv 1 \text{ on } A}.
  \]
  The infimum is achieved for a unique function $h$. If
  $f \in \mathcal{H}$ also satisfies $f \equiv 1$ on $A$, then
  $\mg{A} = \inprod{f}{h}_{\mathcal{H}}$.
\end{thm}

\begin{proof}
  First observe that if $w \in \R^B$ for a finite subset
  $B \subseteq X$, and $f_w = \sum_{b \in B} w_b e^{-d(\cdot, b)}$,
  then
  \begin{equation}
    \label{E:H-norm-weighting}
    w^\transp Z_B w = \sum_{a,b \in B} w_a e^{-d(a,b)} w_b =
    \norm{f_w}_{\mathcal{H}}^2.
  \end{equation}
  
  Now suppose that $\mg{A} < \infty$. If $B \subseteq A$ is finite and
  $w \in \R^B$, then by Proposition \ref{T:finite-sup},
  \eqref{E:H-norm-weighting}, and the definition of $\mg{A}$,
  \[
  \biggl(\sum_{b \in B} w_b\biggr)^2 \le \mg{B}
  \norm{f_w}_{\mathcal{H}}^2 \le \mg{A} \norm{f_w}_{\mathcal{H}}^2.
  \]
  Thus the linear functional \( f_w \mapsto \sum_{b\in B} w_b \) on
  the subspace $\Set{f_w}{w \in \R^B, \ B \subseteq A \text{ finite}}
  \subseteq \mathcal{H}$ has norm at most $\sqrt{\mg{A}}$. Therefore
  there is a function $h \in \mathcal{H}$ with
  $\norm{h}_{\mathcal{H}}^2 = \mg{A}$ such that
  \[
  \sum_{b \in B} w_b = \inprod{f_w}{h}_{\mathcal{H}} = \sum_{b \in B}
  w_b h(b)
  \]
  for every $f_w$; taking $f_w = e^{-d(\cdot, a)}$ for $a \in A$
  yields $h(a) = 1$.

  Next suppose that there exists an $h \in \mathcal{H}$ such that
  $h \equiv 1$ on $A$. Then for any finite subset $B \subseteq A$ and
  $w \in \R^B$, by the Cauchy--Schwarz inequality,
  \[
  \biggl\vert\sum_{b\in B} w_b\biggr\vert = \abs{\inprod{h}{f_w}}
  \le \norm{h}_{\mathcal{H}} \norm{f_w}_{\mathcal{H}}.
  \]
  Equation \eqref{E:H-norm-weighting} and Proposition 
  \ref{T:finite-sup} then imply that $\mg{B} \le
  \norm{h}_{\mathcal{H}}^2$, and so by definition $\mg{A} \le
  \norm{h}_{\mathcal{H}}^2$.

  The above arguments prove both the ``if and only if'' statement and
  the infimum expression for $\mg{A}$.  The last two statements
  follow from elementary Hilbert space geometry.
\end{proof}

We will call the unique function $h$ which achieves the infimum in
Theorem \ref{T:inf} the \dfn{potential function} of $A$.  Theorem
\ref{T:inf} will prove its worth in sections \ref{S:Fourier} and
\ref{S:Euclidean} below.

For now, we consider what has happened to weightings, which were
central to the original category-inspired definition of magnitude, but
have vanished from the scene in Theorem \ref{T:inf}.  Weightings of
finite subspaces of $X$ are naturally identified with elements of the
dual space $\mathcal{H}^*$, if we restrain ourselves from the usual
impulse to identify $\mathcal{H}^*$ with $\mathcal{H}$ itself.  We can
then identify a weighting of a compact subspace $A$ with finite
magnitude as an element of $\mathcal{H}^*$, specifically the element
of $\mathcal{H}^*$ represented by the potential function $h$. See
\cite{MeckMDC} for details.

\subsection{Weight measures}
\label{S:weight-measures}

Proposition \ref{T:semicont} may justify the definition of magnitude
adopted in the previous section as the canonical correct definition,
but it has two deficiencies.  First, it applies only to positive
definite spaces, and second, it lies quite far from the original
category-inspired definition, being fundamentally based instead on the
reformulation in Proposition \ref{T:finite-sup}.  The second drawback
is to some extent addressed in the last paragraph of the previous
section, though still only for positive definite spaces.

In this section we discuss another approach to defining magnitude
for compact metric spaces, first used in \cite{Willerton-homog}, which
more closely parallels the original definition for finite spaces.

A \dfn{weight measure} on a compact metric space $(A,d)$ is a finite
signed Borel measure $\mu$ on $A$ such that
\[
\int_A e^{-d(a,b)} \ d\mu(b) = 1
\]
for every $a \in A$.

A finite metric space $A$ possesses a weight measure $\mu$ if and only
if it possesses a weighting $w \in \R^A$, with the correspondence
given by $w_a = \mu(\{a\})$. The magnitude of $A$ is in that case
\[
\mg{A} = \sum_{a \in A} w_a = \mu(A).
\]
This suggests defining the magnitude of a compact metric space to be
$\mg{A} = \mu(A)$ whenever $A$ possesses a weight measure $\mu$.  The
following result shows that doing so agrees with the definition
adopted in the previous section, whenever both definitions apply.

\begin{prop}[{\cite[Theorem 2.3]{MeckPDM}}]
  \label{T:weight-measure-mag}
  Suppose that $A$ is a compact positive definite metric space with
  weight measure $\mu$.  Then $\mg{A} = \mu(A)$.
\end{prop}

\begin{proof}
  For any finite signed measure $\mu$ on $A$ and $f \in \mathcal{H}$,
  \[
  \abs{\int f \ d\mu} \le \norm{f}_{\infty} \norm{\mu}_{TV} \le
  \norm{f}_{\mathcal{H}} \norm{\mu}_{TV}
  \]
  by \eqref{E:RKHS-bound} (since $K(x,x) = 1$ here), where
  $\norm{\mu}_{TV}$ denotes the total variation norm of $\mu$.
  Therefore $f \mapsto \int f \ d\mu$ is a bounded linear functional
  on $\mathcal{H}$, represented by some $g \in \mathcal{H}$. So for
  each $a \in A$,
  \[
  1 = \int e^{-d(a,b)} \ d\mu(b) = \inprod{e^{-d(\cdot,b)}}{g}_{\mathcal{H}} = g(a).
  \]
  Then by the last statement of Theorem \ref{T:inf}, if $h$ is the
  potential function of $A$, then
  \[
  \mg{A} = \inprod{g}{h}_{\mathcal{H}} = \int h \ d\mu = \mu(A).
  \qedhere
  \]
\end{proof}

In fact it can be shown that $g = h$ in the proof above. 

We therefore define the \dfn{magnitude} of a compact metric
space $A$ with a weight measure $\mu$ to be $\mg{A} := \mu(A)$, with
Proposition \ref{T:weight-measure-mag}'s assurance that when $A$ is
positive definite, this definition is consistent with the previous one.
 
A first nontrivial example is a compact interval $[a,b] \subseteq \R$.
A straightforward computation (see \cite[Theorem
2]{Willerton-homog}) shows that
\begin{equation}
  \label{E:interval-weight-measure}
  \mu_{[a,b]} = \frac{1}{2} (\delta_a + \lambda_{[a,b]} + \delta_b)
\end{equation}
is a weight measure for $[a,b]$, where $\delta_x$ denotes the point
mass at $x$ and $\lambda_{[a,b]}$ denotes Lebesgue measure restricted
to $[a,b]$. It follows that
\begin{equation}
  \label{E:interval-magnitude}
  \mg{[a,b]} = 1 + \frac{b-a}{2}.
\end{equation}
See \cite{SchaWLP} for a contention that (up to the $\frac{1}{2}$
scaling factor) this is the ``correct'' size of an interval. In any
case, the appearance of the length $(b-a)$ gives the first compelling
evidence that magnitude knows about genuinely ``geometric''
information for infinite spaces.

The following easy consequence of Fubini's theorem further extends the
reach of Propositions~\ref{T:finite-ell1-product} and
\ref{T:compact-ell1-product}.

\begin{prop}
  \label{T:weight-measure-product}
  If $\mu_A$ and $\mu_B$ are weight measures on compact metric spaces
  $A$ and $B$, then $\mu_A \otimes \mu_B$ is a weight measure on $A
  \times_1 B$, and so $\mg{A \times_1 B} = \mg{A} \mg{B}$.
\end{prop}

The chief drawback to the definition of magnitude in terms of weight
measures is that many interesting spaces do not possess weight
measures.  For example, the results of \cite{BaCa} imply that balls in
$\ell_2^3$ do not possess weight measures (rather, their weightings
turn out to be higher-order distributions), and numerical computations
in \cite{Willerton-heuristic} suggest that squares and discs in
$\ell_2^2$ also do not possess weight measures.

On the other hand, the following result can be interpreted as saying
that compact positive definite spaces ``almost'' possess weight
measures.

\begin{prop}[{\cite[Theorems 2.3 and 2.4]{MeckPDM}}]
  \label{T:compact-sup}
  If $A$ is a compact positive definite metric space, then
  \[
  \mg{A} = \sup \Set{\frac{\mu(A)^2}{\int_A \int_A e^{-d(a,b)} \
      d\mu(a) \ d\mu(b)}}{ \mu \in M(A), \ \int_A \int_A e^{-d(a,b)} \
      d\mu(a) \ d\mu(b) \neq 0},
  \]
  where $M(A)$ denotes the space of finite signed Borel measures on
  $A$.  The supremum is attained if and only if $A$ possesses a weight
  measure; in that case it is attained precisely by scalar multiples
  of weight measures.
\end{prop}

One positive result about the existence of weight measures is the
following.

\begin{prop}[{\cite[Lemma 2.8 and Corollary 2.10]{MeckPDM}}]
  \label{T:positively-weighted}
  Suppose $(A,d)$ is a compact positive definite space, and that each
  finite $A'\subseteq A$ possesses a weighting with positive
  components. Then $A$ possesses a positive weight measure.
\end{prop}

The hypothesis of Proposition \ref{T:positively-weighted} is
satisfied, for example, by all compact subsets of $\R$ and by all
compact ultrametric spaces (see Theorem \ref{T:R-finite-mag} below and
\cite[Proposition 2.4.18]{MMS}).  Since Proposition
\ref{T:positively-weighted} applies only to positive definite spaces,
it does not extend the scope of magnitude beyond that of the previous
section.  Nevertheless, the existence of a positive weight measure
makes it much easier to compute magnitude, and has other theoretical
consequences which will come up in the next section.

The following generalization of Proposition \ref{T:Speyer} gives
another large class of spaces which possess weight measures.

\begin{lemma}[{\cite[Theorem 1]{Willerton-homog}}]
  \label{T:homog-weight-measure}
  Let $A$ be a compact homogeneous metric space. Then $A$ possesses a
  weight measure, which is a scalar multiple of the unique
  isometry-invariant probability measure $\mu$ on $A$.  Furthermore,
  \[
  \mg{A} = \left(\int_A \int_A e^{-d(a,b)} \ d\mu(a) \
    d\mu(b)\right)^{-1}.
  \]
\end{lemma}

Using Lemma \ref{T:homog-weight-measure}, Willerton explicitly
computed the magnitudes of round spheres with the geodesic metric: for
$n$ even, the magnitude of the $n$-sphere with radius $R$ is
\[
\frac{2}{1 + e^{-\pi R}}
\left[ 1 + \left( \frac{R}{1} \right)^2 \right]
\left[ 1 + \left( \frac{R}{3} \right)^2 \right]
\cdots
\left[ 1 + \left( \frac{R}{n - 1} \right)^2 \right],
\]
and there is a similar formula for odd $n$; see \cite[Theorem
7]{Willerton-homog}.

Lemma \ref{T:homog-weight-measure} is particularly useful in analyzing
the magnitude function of a homogeneous space $A$, since it implies
that $tA$ possesses a weight measure for every $t > 0$, which is
moreover independent of $t$ (up to normalization). In the particular
case of a homogeneous Riemannian manifold, Willerton proved the
following asymptotic results. (We note that most homogeneous manifolds
are \emph{not} of negative type, so that $tA$ need not be positive
definite; see \cite{Kokkendorff}.)

\begin{thm}[{\cite[Theorem 11]{Willerton-homog}}]
  \label{T:homog-manifold}
  Suppose that $(M,d)$ is an $n$-dimensional homogeneous Riemannian
  manifold equipped with its geodesic distance $d$. Then
  \[
  \mg{tM} = \frac{1}{n! \omega_n} \left(\vol(M) t^n + \frac{n+1}{6}
    \operatorname{tsc}(M) t^{n-2} + O(t^{n-4})\right)
    \quad \text{as } t \to \infty,
  \]
  where $\vol$ denotes Riemannian volume, $\operatorname{tsc}$ denotes
  total scalar curvature, and $\omega_n$ is the volume of the
  $n$-dimensional unit ball in $\ell_2^n$.

  In particular, if $M$ is a homogeneous Riemannian surface, then
  \[
  \mg{tM} = \frac{\operatorname{area} (M)}{2\pi} t^2 + \chi(M) +
  O(t^{-2}) \quad \text{as } t \to \infty,
  \]
  where $\chi(M)$ denotes the Euler characteristic of $M$.
\end{thm}

Theorem \ref{T:homog-manifold} shows in particular that the magnitude
function of a homogeneous Riemannian manifold determines both its
volume and its total scalar curvature.

We note that most Riemannian manifolds are neither homogeneous nor
positive definite, and it is so far not clear how to define their
magnitude.

\subsection{Maximum diversity}
\label{S:diversity}

Proposition \ref{T:compact-sup} suggests considering, for a compact
metric space $(A,d)$, the quantity
\begin{equation}
  \label{E:md}
  \begin{split}
    \md{A} := & \sup \Set{\frac{\mu(A)^2}{\int_A \int_A e^{-d(a,b)} \
        d\mu(a) \ d\mu(b)}}{ \mu \in M_+(A), \mu \neq 0} \\
    = & \sup_{\mu \in P(A)} \left(\int_A \int_A e^{-d(a,b)} \
      d\mu(a) \ d\mu(b)\right)^{-1},
  \end{split}
\end{equation}
where $M_+(A)$ is the space of finite positive Borel measures on $A$,
and $P(A)$ is the space of Borel probability measures on $A$.  We
refer to $\md{A}$ as the \dfn{maximum diversity} of $A$, for reasons
that will be described shortly.  Maximum diversity lacks the
category-theoretic motivation of magnitude, but it turns out to have
its own interesting interpretations, and to be both intimately related
to magnitude and easier to analyze in certain respects.

Regarding interpretation, suppose that $A$ is finite, the points of
$A$ represent species in some ecosystem, and that
$e^{-d(a,b)} \in (0,1]$ represents the ``similarity'' of two species
$a,b \in A$. If $\mu \in P(A)$ gives the relative abundances of
species, then
\[
\left(\int_A \int_A e^{-d(a,b)} \ d\mu(a) \ d\mu(b)\right)^{-1}
\]
gives a way of quantifying the ``diversity'' of the ecosystem which is
sensitive to both the abundances of the species and the similarities
between them; see \cite{LeCo} for extensive discussion of
a much larger family of diversities that this fits into. It is this
interpretation that motivates the name ``maximum diversity''.

There are multiple connections between magnitude and maximum
diversity.  The most obvious is that, by Proposition
\ref{T:compact-sup}, $\md{A} \le \mg{A}$ for any compact positive
definite space $A$.  Moreover, Proposition \ref{T:compact-sup} implies
that $\md{A} = \mg{A}$ if $A$ is positive definite and possesses a
positive weight measure; Proposition \ref{T:positively-weighted} and
Lemma \ref{T:homog-weight-measure} indicate some families of such
spaces.  Finally, as we will see in Corollary
\ref{T:mag-md-equivalence} below, if $A \subseteq \ell_2^n$, then the
inequality $\md{A} \le \mg{A}$ can be reversed, up to a
(dimension-dependent) multiplicative constant.  We will see
applications of all these connections below.

A more subtle connection between maximum diversity and magnitude,
which we will not discuss here, is proved in the main result of
\cite{METAMB,LeMe}.

We now move on to ways in which maximum diversity is better behaved
than magnitude.  One is that the supremum in \eqref{E:md} is always
achieved, unlike the one in Proposition \ref{T:compact-sup}. This is a
consequence of the compactness of $P(A)$ in the weak-$*$ topology; see
\cite[Proposition 2.9]{MeckPDM} (this fact is used in the proof of
Proposition \ref{T:positively-weighted} above). Another is the
following improvement, for maximum diversity, of Proposition
\ref{T:semicont}.

\begin{prop}[{\cite[Proposition 2.11]{MeckPDM}}]
  \label{T:md-cont}
  The maximum diversity $\md{A}$ is continuous as a function of
  $A$, on the class of compact metric spaces equipped with the
  Gromov--Hausdorff topology. 
\end{prop}

\begin{cor}[{\cite[Corollary 2.12]{MeckPDM}}]
  \label{T:pw-mag-cont}
  The magnitude $\mg{A}$ is continuous as a function of $A$, on the
  class of compact positive definite metric spaces which possess
  positive weight measures, equipped with the Gromov--Hausdorff
  topology.

  In particular, magnitude is continuous on the class of compact
  subsets of $\R$, and on the class of compact ultrametric spaces.
\end{cor}


The next result shows how the asymptotic behavior of $\mg{tA}_+$ is
relatively easy to analyze.  Recall that the \dfn{covering number}
$N(A,\eps)$ is the minimum number of $\eps$-balls required to
cover $A$, and that the \dfn{Minkowski dimension} of $A$ may be
defined as
\begin{equation}
  \label{E:boxdim}
  \boxdim A := \lim_{\eps \to 0^+} \frac{N(A, \eps)}{\log(1/\eps)}
\end{equation}
whenever this limit exists. The idea of the proof of Proposition
\ref{T:ddim} below is simply that when $t$ is large and $\eps$ is
small, the supremum over $P(A)$ defining $\mg{tA}_+$ is approximately
attained by a measure uniformly supported on the centers of a maximal
family of disjoint $\eps$-balls in $A$.

\begin{prop}[{\cite[Theorem 7.1]{MeckMDC}}]
  \label{T:ddim}
  If $A$ is a compact metric space, then
  \begin{equation}
    \label{E:ddim}
    \lim_{t \to \infty} \frac{\log \mg{tA}_+}{\log t} = \boxdim A.
  \end{equation}
\end{prop}

Proposition \ref{T:ddim} should be interpreted as saying that the
limit on the left hand side of \eqref{E:ddim} exists if and only if
$\boxdim A$ exists.  Moreover, if the limit is replaced with a lim inf
or lim sup, the left hand side of \eqref{E:ddim} is equal to the
so-called lower or upper Minkowski dimension of $A$, respectively,
defined by modifying \eqref{E:boxdim} in the same way.

Since $\md{A} \le \mg{A}$ for any compact positive definite space,
Proposition \ref{T:ddim} gives a lower bound for the growth rate of
the magnitude function for a compact space of negative type. Moreover,
in Euclidean space $\ell_2^n$, Proposition 3.12 and the rough
equivalence of magnitude and maximum diversity mentioned above will be
used to show that Minkowski dimension can be recovered from magnitude;
see Theorem \ref{T:mdim} below.  (Proposition 7.5 of \cite{MeckMDC}
proves the same fact for compact homogeneous metric spaces, using
Lemma \ref{T:homog-weight-measure} above.)

\section{Magnitude in normed spaces}
\label{S:norms}

In this section we will specialize magnitude to compact subsets of
finite-dimensional vector spaces with translation-invariant metrics.
It is in these settings that we find the strongest connections between
magnitude and geometry.  In section \ref{S:R}, we find a quite
complete description of the magnitude of an arbitrary compact set
$A \subseteq \R$; in particular, $\mg{A}$ depends only on the Lebesgue
measure of $A$ and the sizes of the ``gaps'' in $A$
(Corollary~\ref{T:R-mag-gaps}).  In section \ref{S:1-norm}, we show
that in $\ell_1^n$, magnitude can be used to recover $\ell_1$
analogues of the classical intrinsic volumes of a convex body
(Theorem~\ref{thm:ell1-cvx}).  In section \ref{S:Fourier}, we apply
Fourier analysis to the study of magnitude, when $\R^n$ is equipped
with a norm (or more generally, a $p$-norm) which makes it a positive
definite metric space. In particular, we find that magnitude is
continuous on convex bodies in such spaces (Theorem
\ref{T:star-continuity}). Finally, in section \ref{S:Euclidean}, we
specialize these tools to the most familiar normed space, the
Euclidean space $\ell_2^n$.  In that setting the Fourier-analytic
perspective of section \ref{S:Fourier} uncovers connections with
partial differential equations and potential theory. Among other
results, we will see that in Euclidean space, magnitude knows about
volume (Theorem~\ref{T:volume}) and Minkowski dimension
(Theorem~\ref{T:mdim}), although there are frustratingly few compact
sets in $\ell_2^n$ whose exact magnitudes are known (see
Theorem~\ref{T:l2-balls}).

Corollary \ref{T:R-mag-gaps} and the material of section
\ref{S:1-norm} are new.  Most of the results of section
\ref{S:Fourier} were previously proved for Euclidean space, but are
new in the generality discussed here.

\subsection{Magnitude in $\R$}
\label{S:R}

In the real line $\R$, magnitude can be analyzed in great detail
thanks to the order structure underlying the metric structure.
Namely, if $a < b < c$, then $Z(a,c) = Z(a,b) Z(b,c)$, where we recall
that $Z(a,b) = e^{-d(a,b)}$.  This simple fact lies behind the proof
of the next result.

\begin{thm}[{\cite[Theorem 4]{LeWi} and \cite[Proposition
    2.4.13]{MMS}}]
  \label{T:R-finite-mag}
  Given real numbers $a_1 < a_2 < \dots < a_N$, the weighting $w$ of $A =
  \{a_1, \dots, a_N\}$ is given by 
  \[
  w_{a_i} = \frac{1}{2} \left( \tanh \frac{a_i - a_{i-1}}{2} + \tanh
    \frac{a_{i+1} - a_i}{2} \right)
  \]
  for $2 \le i \le N-1$, and 
  \[
  w_{a_1} = \frac{1}{2} \left( 1 + \tanh \frac{a_2 - a_1}{2}\right),
  \qquad
  w_{a_N} = \frac{1}{2} \left( 1 + \tanh \frac{a_N - a_{N-1}}{2} \right).
  \]
  Consequently,
  \[
  \mg{A} = 1 + \sum_{i=2}^{N} \tanh \frac{a_i - a_{i-1}}{2}.  
  \]
\end{thm}

Theorem \ref{T:R-finite-mag}, together with Proposition
\ref{T:finite-approximation}, was used to give the first derivation of
the magnitude of an interval; see \cite[Theorem 7]{LeWi} and
\cite[Theorem 3.2.2]{MMS}.

As mentioned above, by Proposition \ref{T:positively-weighted},
Theorem \ref{T:R-finite-mag} implies that every compact subset of $\R$
possesses a weight measure.  Furthermore, as noted in Corollary
\ref{T:pw-mag-cont}, this implies that magnitude on $\R$ is continuous
with respect to the Gromov--Hausdorff topology.

The last part of the following corollary appears, with additional
technical assumptions, as \cite[Lemma 3]{Willerton-homog}.

\begin{cor}
  \label{T:R-union-mag}
  Suppose that $A, B \subseteq \R$ are compact with $a = \sup A \le
  \inf B = b$. Then
  \[
  \mg{A \cup B} = \mg{A} + \mg{B} - 1 + \tanh \frac{b - a}{2}. 
  \]
  Consequently, if $C \subseteq \R$ is compact and $[a,b] \subseteq
  C$, then
  \[
  \mg{C \setminus (a,b)} = \mg{C} - \frac{b-a}{2} + \tanh
  \frac{b-a}{2}.
  \]
\end{cor}

\begin{proof}
  The first claim follows immediately from Theorem
  \ref{T:R-finite-mag} in the case that $A$ and $B$ are finite, and
  then follows for general compact sets by continuity.  The second
  equality follows by writing $C = A \cup [a,b] \cup B$, where
  $A = C \cap (-\infty, a]$ and $B = C \cap [b, \infty)$, then
  applying the first equality twice and \eqref{E:interval-magnitude}.
\end{proof}

Corollary \ref{T:R-union-mag}, together with continuity and the
knowledge of the magnitude of a compact interval, can be used to
compute the magnitude of any compact set $A \subseteq \R$, since $A$
can be written as
\begin{equation}
  \label{E:R-compact}
  A = [a,b] \setminus \bigcup_i (a_i, b_i),
\end{equation}
where $\{(a_i,b_i)\}$ is a finite or countable collection of disjoint
subintervals of $[a,b]$. 

\begin{cor}
  \label{T:R-mag-gaps}
  If $A \subseteq \R$ is compact, then
  \[
  \mg{A} = 1 + \frac{\vol_1 A}{2} + \sum_i \tanh \frac{b_i - a_i}{2},
  \]
  where $a_i$ and $b_i$ are as in \eqref{E:R-compact}.
\end{cor}

Another proof of Corollary \ref{T:R-mag-gaps} can be given using
\cite[Proposition 3.2.3]{MMS}. As an application of Corollary
\ref{T:R-mag-gaps}, we obtain the magnitude of the length $\ell$
ternary Cantor set $C_{\ell}$ (see \cite[Theorem
10]{LeWi}, \cite[Theorem 4]{Willerton-homog}):
\[
\mg{C_\ell} = 1 + \frac{1}{2} \sum_{i=1}^\infty \tanh \frac{\ell}{2
  \cdot 3^i}.
\]

\subsection{Magnitude in the $\ell_1$-norm}
\label{S:1-norm}

The magnitude of subsets of $\R^n$ is generally most tractable when we
equip $\R^n$ with the $\ell_1$-norm.  Although that may not be the
norm of primary geometric interest, it provides a testing ground for
questions that are more difficult to settle in Euclidean space.

We have already seen that $\ell_1^n$, like $\ell_2^n$, is of negative
type (Theorem \ref{T:pd-examples}).  The key difference is
Proposition \ref{T:compact-ell1-product}, the multiplicativity of
magnitude with respect to the $\ell_1$ product.  Since we already know
the magnitude of intervals, this immediately allows us to calculate
the magnitude of boxes in $\ell_1^n$.  Unions of boxes can then be
used to approximate more complex subsets, as we shall see.

Explicitly, a box $\prod_{i = 1}^n [a_i, a_i + L_i]$ in $\ell_1^n$ has
magnitude $\prod_{i = 1}^n (1 + L_i/2)$.  It follows that
$\mg{tA} \to 1$ as $t \to 0^+$ for boxes $A$.  But then monotonicity
of magnitude (Proposition \ref{T:compact-monotone}) implies a more
general result:

\begin{prop}
\label{propn:ell1-mag-at-zero}
If $A \subseteq \ell_1^n$ is compact, then $\lim_{t \to 0^+}
\mg{tA} = 1$.
\end{prop}
(In $\ell_2^n$, this is much harder to prove; see Theorem \ref{T:shrink}.)
Proposition~\ref{propn:ell1-mag-at-zero} and Theorem~\ref{T:mf-cont}
together imply that the magnitude function $t \mapsto \mg{tA}$ is
continuous on $[0, \infty)$.

Our formula for the magnitude of a box in $\ell_1^n$ can be rewritten in
terms of the intrinsic volumes $V_0, V_1, \ldots$ (defined in, for
instance, Chapter~7 of \cite{KlRo} or Chapter~4 of \cite{SchnCBB}).
Recall that $V_i(A)$ is the canonical $i$-dimensional measure of a convex
set $A \subseteq \R^n$, and that the intrinsic volumes are characterized by
Steiner's polynomial formula
\[
\vol(A + r\ball{n}) = \sum_{i = 0}^n \omega_{n - i} V_i(A) r^{n - i}
\]
(Proposition~9.2.2 of \cite{KlRo} or Equation~4.1 of \cite{SchnCBB}),
where $\ball{n}$ is the unit Euclidean $n$-ball and
$\omega_j = \vol(\ball{j})$.  For boxes $A \subseteq \ell_1^n$, the
formula above can be rewritten as
\begin{equation}
\label{eq:ell1-mag-box-bad}
\mg{A} 
=
\sum_{i = 0}^n \frac{V_i(A)}{2^i},
\end{equation}
either by direct calculation or by noting that $\mg{[0, L]} = 1 + V_1([0,
  L])/2$ and using the multiplicative property of the intrinsic volumes
(Theorem~9.7.1 of \cite{KlRo}).  Hence the magnitude function of a box
$A$ is a polynomial
\[
\mg{tA}
=
\sum_{i = 0}^n \frac{V_i(A)}{2^i} t^i
\]
whose coefficients are (up to known factors) the intrinsic volumes of $A$,
and whose degree is its dimension.  In particular, the magnitude function
of a box determines all of its intrinsic volumes and its dimension.


In fact, such a result is true for a much larger class of subsets of
$\ell_1^n$ than just boxes.  To show this, we must adapt the classical
notion of intrinsic volume to $\ell_1^n$, following \cite{IGON}.

First recall that a metric space $A$ is \dfn{geodesic} if for any $a, b
\in A$ there exists a distance-preserving map $\gamma \colon [0, d(a, b)]
\to A$ such that $\gamma(0) = a$ and $\gamma(d(a, b)) = b$.  The geodesic
subsets of $\ell_2^n$ are the convex sets.  The geodesic subsets of
$\ell_1^n$, called the \dfn{$\ell_1$-convex sets} \cite{IGON}, include the
convex sets and much else besides (such as \textsf{L} shapes).  In this
setting, there is a Steiner-type theorem in which balls are replaced by
cubes (Theorem~6.2 of \cite{IGON}): for any $\ell_1$-convex compact set $A
\subseteq \ell_1^n$, writing $\cube{n} = [-1/2, 1/2]^n$,
\begin{equation}
\label{eq:steiner}
\vol(A + r\cube{n}) 
= 
\sum_{i = 0}^n V'_i(A) r^{n - i}
\end{equation}
where $V'_0(A), \ldots, V'_n(A)$ depend only on $A$.

The functions $V'_0, V'_1, \ldots$ on the class of $\ell_1$-convex
compact sets are called the \dfn{$\ell_1$-intrinsic volumes}
\cite{IGON}.  They are valuations (that is, finitely additive),
continuous with respect to the Hausdorff metric, and invariant under
isometries of $\ell_1^n$.  There is a well-developed integral geometry
of $\ell_1$-convex sets \cite{IGON}, closely parallel to the classical
integral geometry of convex sets; for instance, there is a
Hadwiger-type theorem for $\ell_1$-intrinsic volumes.

Although the intrinsic and $\ell_1$-intrinsic volumes are not in general
equal, they coincide for boxes $A$, giving
\begin{equation}
\label{eq:mag-ell1-box}
\mg{A} = \sum_{i = 0}^n \frac{V'_i(A)}{2^i},
\qquad
\mg{tA} = \sum_{i = 0}^n \frac{V'_i(A)}{2^i} t^i
\end{equation}
(the latter because $V'_i$ is homogeneous of degree $i$).  It is this
relationship, not~\eqref{eq:ell1-mag-box-bad}, that generalizes
from boxes to a much larger class of sets.

\begin{conj}[{\cite[Conjecture~3.4.10]{MMS}}]
  \label{C:ell1-convex}
  For all compact $\ell_1$-convex sets $A \subseteq \ell_1^n$,
  \[
  \mg{A} = \sum_{i = 0}^n \frac{V'_i(A)}{2^i}.
  \]
\end{conj}

We will prove the following parts of this conjecture:

\begin{thm}
\label{thm:ell1-cvx}
~
\begin{enumerate}
\item
\label{part:ell1-cvx-ineq}
$\mg{A} \leq \sum_{i = 0}^n 2^{-i} V'_i(A)$ for all compact
  $\ell_1$-convex sets $A \subseteq \ell_1^n$.

\item
\label{part:ell1-cvx-bodies}
$\mg{A} = \sum_{i = 0}^n 2^{-i} V'_i(A)$ for all convex bodies $A
\subseteq \ell_1^n$.

\item
\label{part:ell1-cvx-two}
$\mg{A} = \sum_{i = 0}^2 2^{-i} V'_i(A)$ for all compact convex sets
$A \subseteq \ell_1^2$.
\end{enumerate}
\end{thm}

(A \dfn{convex body} is a compact convex set with nonempty interior.)

For the proof, we will use some special classes of box.  A \dfn{pixel} in
$\R^n$ is a unit cube $\prod_{i = 1}^n [a_i, a_i + 1]$ with integer
coordinates $a_i$.  More generally, a \dfn{subpixel} is a box $\prod_{i =
  1}^n [a_i, b_i]$ with $a_i \in \mathbb{Z}$ and $b_i \in \{a_i, a_i +
1\}$.  Note that the intersection of two subpixels is either a subpixel or
empty.

Equation \eqref{E:interval-weight-measure} and Proposition
\ref{T:weight-measure-product} imply that for any box $B = \prod_i
[a_i, b_i]$ in $\ell_1^n$, the product measure $\mu_B = \prod_{i =
  1}^n \mu_{[a_i, b_i]}$ is a weight measure on $B$.

\begin{lemma}
\label{lemma:subpixelated-extension}
There is a unique function
\[
\begin{array}{ccc}
\{\text{finite unions of subpixels in } \R^n \}     &
\to     &
\{\text{signed Borel measures on } \R^n \}      \\
A       &
\mapsto &
\mu_A
\end{array}
\]
extending the definition above for subpixels and satisfying $\supp
\mu_A \subseteq A$, $\mu_\emptyset = 0$, and $\mu_{A \cup B} = \mu_A +
\mu_B - \mu_{A \cap B}$ whenever $A$ and $B$ are finite unions of
subpixels.
\end{lemma}

\begin{proof}
  By the extension theorem of Groemer (Theorem~6.2.1 of
  \cite{SchnCBB}), it suffices to show that for any subpixels
  $B_1, \ldots, B_m$ such that $B_1 \cup \cdots \cup B_m$ is a
  subpixel,
\[
\mu_{B_1 \cup \cdots \cup B_m}
=
\sum_{k \geq 0} (-1)^k \!\!
\sum_{1 \leq j_0 < \cdots < j_k \leq m} 
\mu_{B_{j_0} \cap \cdots \cap B_{j_k}}.
\]
But $B_1 \cup \cdots \cup B_m$ is only a subpixel if some $B_j$ contains
all the others, and in that case the sum telescopes and the proof is
trivial.
\end{proof}

A subset $A$ of $\ell_1^n$ is \dfn{1-pixelated} if it is a finite union
of pixels; then $\lambda A$ is said to be \dfn{$\lambda$-pixelated}.  A
set is \dfn{pixelated} if it is $\lambda$-pixelated for some $\lambda >
0$.

\begin{prop}
\label{propn:ell1-wtg}
Let $A \subseteq \ell_1^n$ be an $\ell_1$-convex pixelated set.  Then
$\mu_A$ as given in Lemma \ref{lemma:subpixelated-extension} is a
weight measure on $A$.
\end{prop}

\begin{proof}
We may harmlessly assume that $A$ is 1-pixelated.  The result holds when
either $n = 0$ or $A$ is a single pixel.  So, we may assume inductively that
$n \geq 1$, that $A$ contains at least two pixels, and that the result
holds for $\ell_1$-convex 1-pixelated sets of smaller dimension or fewer
pixels than $A$.

Fix $a \in A$.  We may assume without loss of generality that at least two
of the pixels in $A$ differ in their last coordinates, that $\sup_{b \in A}
b_n = 1$, and that $a$ belongs to some pixel of $A$ whose center has
negative last coordinate.  Write $A_-$ for the union of the pixels in $A$
whose centers have negative last coordinates, and similarly $A_+$.  Thus,
$a \in A_-$ and the center of every pixel in $A_+$ has last coordinate
$1/2$.  Both $A_-$ and $A_+$ are $\ell_1$-convex 1-pixelated sets (by
Lemma~3.3 of \cite{IGON}), and $A_- \cap A_+$ is a finite union of subpixels
(though need not be pixelated).

We have to show that
\[
\int_{\R^n} Z(a, b) \dee\mu_A(b) = 1.
\]
Since $\mu_A = \mu_{A_+} + \mu_{A_-} - \mu_{A_+ \cap A_-}$ and $\mu_{A_-}$
is a weight measure on $A_-$ (by inductive hypothesis), an equivalent
statement is that
\begin{equation}
\label{eq:ell1-int-1}
\int_{\R^n} Z(a, b) \dee\mu_{A_+}(b)
=
\int_{\R^n} Z(a, b) \dee\mu_{A_- \cap A_+}(b).
\end{equation}
Write $\pi \colon \R^n \to \R^{n - 1}$ for orthogonal projection onto the
first $(n - 1)$ coordinates, and write $a' = (\pi(a), 0) = (a_1, \ldots,
a_{n - 1}, 0)$.  Then $Z(a, b) = Z(a, a') Z(a', b)$ for $b \in A_+$,
so~\eqref{eq:ell1-int-1} is equivalent to
\[
\int_{\R^n} Z(a', b) \dee\mu_{A_+}(b)
=
\int_{\R^n} Z(a', b) \dee\mu_{A_- \cap A_+}(b).
\]
We analyze each side in turn.  First, $A_+ = (\pi A_+) \times [0, 1]$,
so it follows from Proposition \ref{T:weight-measure-product}
that $\mu_{A_+} = \mu_{\pi A_+} \otimes \mu_{[0, 1]}$.  Using this and
the fact that $\mu_{[0, 1]}$ is a weight measure on $[0, 1]$, we find
that the left-hand side is equal to
\begin{equation}
\label{eq:int-lhs}
\int_{\R^{n - 1}} Z(\pi(a), c) \dee\mu_{\pi A_+}(c).
\end{equation}
Next, $\mu_{A_- \cap A_+}$ is supported on $\R^{n - 1} \times \{0\}$,
and $\pi(A_- \cap A_+) = \pi A_- \cap \pi A_+$ (by Corollary~2.5 of
\cite{IGON}), which together imply that the right-hand side is equal
to
\begin{equation}
\label{eq:int-rhs}
\int_{\R^{n - 1}} Z(\pi(a), c) \dee\mu_{\pi A_- \cap \pi A_+}(c).
\end{equation}
Hence it suffices to show that the integrals~\eqref{eq:int-lhs}
and~\eqref{eq:int-rhs} are equal.  Since $\mu_{\pi A} = \mu_{\pi A_-} +
\mu_{\pi A_+} - \mu_{\pi A_- \cap \pi A_+}$, an equivalent statement is
that 
\begin{equation}
\label{eq:ell1-int-2}
\int_{\R^{n - 1}} Z(\pi(a), c) \dee\mu_{\pi A}(c)
=
\int_{\R^{n - 1}} Z(\pi(a), c) \dee\mu_{\pi A_-}(c).
\end{equation}
But $\pi A$ and $\pi A_-$ are 1-pixelated sets of dimension $n - 1$,
and are $\ell_1$-convex (by Corollary~1.12 of \cite{IGON}), so our
inductive hypothesis implies that $\mu_{\pi A}$ and $\mu_{\pi A_-}$
are weight measures on them.  Since
$\pi(a) \in \pi A_- \subseteq \pi A$, both sides
of~\eqref{eq:ell1-int-2} are equal to $1$, completing the proof.
\end{proof}

Our proof of Theorem~\ref{thm:ell1-cvx} rests on the following
result: 

\begin{prop}
\label{propn:ell1-cvx-pix}
$\mg{A} = \sum_{i = 0}^n 2^{-i} V'_i(A)$ for all pixelated
  $\ell_1$-convex sets $A \subseteq \ell_1^n$. 
\end{prop}

\begin{proof}
  Assume that $A$ is 1-pixelated, and write $A$ as a union
  $\bigcup_{j = 1}^m B_j$ of pixels.  Also write
  $W = \sum_{i = 0}^n 2^{-i} V'_i$; then $\mg{B} = W(B)$ whenever
  $B$ is a box or the empty set.
  Propositions~\ref{T:weight-measure-mag} and~\ref{propn:ell1-wtg}
  together with the valuation property of $W$ give
\begin{align*}
\mg{A}  
=
\mu_A(\R^n)   &  
=
\sum_{k \geq 0} (-1)^k 
\sum_{1 \leq j_0 < \cdots < j_k \leq m} 
\mu_{B_{j_0} \cap \cdots \cap B_{j_k}} (\R^n)   \\
&
=
\sum_{k \geq 0} (-1)^k 
\sum_{1 \leq j_0 < \cdots < j_k \leq m} 
\mg{B_{j_0} \cap \cdots \cap B_{j_k}}   \\
&
=
\sum_{k \geq 0} (-1)^k 
\sum_{1 \leq j_0 < \cdots < j_k \leq m} 
W(B_{j_0} \cap \cdots \cap B_{j_k})   
=
W(A),
\end{align*}
as required.
\end{proof}

\begin{proof}[Proof of Theorem~\ref{thm:ell1-cvx}]
  For part~(\ref{part:ell1-cvx-ineq}), let $A \subseteq \ell_1^n$ be a
  compact $\ell_1$-convex set.  For each $\lambda > 0$, let
  $A_\lambda$ be the smallest $\lambda$-pixelated set containing $A$.
  Then $A_\lambda$ is $\ell_1$-convex (by Proposition~3.1 of
  \cite{IGON}), and $A_\lambda \to A$ in the Hausdorff metric as
  $\lambda \to 0$.  The result now follows from
  Proposition~\ref{propn:ell1-cvx-pix}, continuity of the
  $\ell_1$-intrinsic volumes, and the monotonicity of magnitude
  (Proposition \ref{T:compact-monotone}).

  For~(\ref{part:ell1-cvx-bodies}), let $A \subseteq \ell_1^n$ be a
  compact convex set with $0$ in its interior.  Given
  $\varepsilon > 0$, we can choose $\alpha < 1$ such that
  $d_H(\alpha A, A) < \varepsilon$.  But by convexity, $\alpha A$ is a
  subset of the interior of $A$, so we can choose $\lambda > 0$ such
  that $\alpha A_\lambda \subseteq A$.  Thus, we have a pixelated
  $\ell_1$-convex subset $B = \alpha A_\lambda$ of $A$ satisfying
  $d_H(B, A) < \varepsilon$.  Arguing as in
  part~(\ref{part:ell1-cvx-ineq}) but approximating from the inside
  rather than the outside, we obtain the opposite inequality
  $\mg{A} \geq \sum \frac{V'_i(A)}{2^i}$.  (Alternatively, use
  Theorem~\ref{T:star-continuity} below.)

  For~(\ref{part:ell1-cvx-two}), the only nontrivial case remaining is
  that of a line segment, which is straightforward.
\end{proof}

\subsection{The Fourier-analytic perspective}
\label{S:Fourier}

In the real line, the study of magnitude is facilitated by the order
structure of $\R$; in $\ell_1^n$ we can exploit the algebraic
structure of $\ell_1$ products.  In general normed spaces the most
obvious special feature is translation-invariance.  It will therefore
come as no surprise that Fourier analysis is our key tool in that
setting. This approach was developed in \cite{MeckMDC} for $\ell_2^n$,
but with some additional effort we can work not only with more general
norms but with the broader class of $p$-(quasi)norms for $0 < p \le 1$.

Let $0 < p \le 1$.  A \dfn{$p$-norm} on a real vector space $V$ is a
function $\norm{\cdot} : V \to \R$ such that
\begin{itemize}
\item $\norm{v} \ge 0$ for every $v \in V$, with equality only if $v =
  0$;
\item $\norm{tv} = \abs{t} \norm{v}$ for every $t \in \R$ and $v \in
  V$;
\item $\norm{v+w}^p \le \norm{v}^p + \norm{w}^p$ for every $v, w \in V$.
\end{itemize}
Thus a $1$-normed space is simply a normed space. A principal example
of a $p$-normed space for $p < 1$ is $L_p[0,1]$ with
$\norm{f} = \bigl(\int_0^1 \abs{f(x)}^p \ dx\bigr)^{1/p}$.

If $(V,\norm{\cdot})$ is a $p$-normed space, then
$d_p(v,w) = \norm{v-w}^p$ is a metric on $V$. Conversely, if $d$ is
any translation-invariant, symmetric, positively homogeneous metric on
a real vector space $V$, then $\norm{v} = d(v,0)$ defines a $p$-norm
on $V$, where $p \in (0,1]$ is the degree of homogeneity of $d$.

The following classical result, which goes back to L\'evy \cite{Levy}
(see also \cite[Theorem 6.6]{Koldobsky}), identifies which
finite-dimensional $p$-normed spaces are positive definite metric
spaces (and hence, by homogeneity, of negative type).

\begin{thm}
  \label{T:pd-p-norms}
  Let $0 < p \le 1$, let $\norm{\cdot}$ be a $p$-norm on $\R^n$, and
  equip $\R^n$ with the metric $d_p(x,y) = \norm{x-y}^p$. Then
  $(\R^n, d_p)$ is a positive definite metric space if and only if
  there is linear map $T:\R^n \to L_p[0,1]$ such that
  $\norm{Tx}_p = \norm{x}$ for every $x \in \R^n$.
\end{thm}

Theorem~\ref{T:pd-p-norms} implies in particular that $L_p[0,1]$ and
$\ell_p^n$ are positive definite with the metric $d_p$ for $0 < p \le
1$.  We recall from Theorem \ref{T:pd-examples} that $L_q[0,1]$ and
$\ell_q^n$ are also positive definite, with the usual metric, for
$1\le q \le 2$.

To simplify the statements of results:
\begin{quote}
  For the rest of this section, $\norm{\cdot}$ will always
  denote a $p$-norm on $\R^n$ such that $(\R^n, d_p)$ is a positive
  definite metric space.
\end{quote}
We will make use of the function $F_p : \R^n \to \R$ defined by
$F_p(x) = e^{-\norm{x}^p}$, and denote by
$\mathcal{B} = \Set{x \in \R^n}{\norm{x} = 1}$ the unit ball of
$\norm{\cdot}$. For $f \in L_1(\R^n)$, we adopt the convention that
the Fourier transform of $f$ is given by
$\widehat{f}(x) = \int_{\R^n} f(y) e^{-2\pi i \inprod{x}{y}} \ dy$.

A key observation is that $F_p$ is the Fourier transform of a
$p$-stable probability distribution. Proposition \ref{T:nice-norms}
collects some crucial facts which follow from results from the
literature on stable random processes.

\begin{prop}
  \label{T:nice-norms}
  ~
  \begin{enumerate}
  \item \label{P:lower-bound} There is a constant $c > 0$ (depending
    on the $p$-norm $\norm{\cdot}$) such that $\widehat{F_p}(x) \ge c
    (1+\norm{x}_2)^{-(1+p)n}$ for every $x \in \R^n$.
  \item \label{P:monotone} For each $x \in \R^n$, $\widehat{F_p}(tx)$
    is nonincreasing as a function of $t \ge 0$. In particular,
    $\bigl\Vert\widehat{F_p}\bigr\Vert_\infty = \widehat{F_p}(0) =
    \Gamma\bigl(\frac{n}{p} + 1\bigr) \vol \mathcal{B}$.
  \end{enumerate}
\end{prop}

\begin{proof}
  It follows from Theorem \ref{T:pd-p-norms} and Bochner's theorem
  that $\widehat{F_p}$ is the density of a $p$-stable distribution
  $\mu$ on $\R^n$.
  
  By a theorem of L\'evy (see \cite[Lemma 6.4]{Koldobsky}), there is a
  symmetric measure $\sigma$ on $S^{n-1}$ such that
  \[
  \norm{x}^p = \int_{S^{n-1}} \abs{\inprod{x}{\theta}}^p \
  d\sigma(\theta);
  \]
  since $\norm{x} \neq 0$ for $x \neq 0$, the support of $\sigma$ is
  not contained in any proper subspace of $\R^n$.  Then $\sigma$ is a
  positive scalar multiple of the spherical part of the L\'evy measure
  of $\mu$ (cf.\ \cite[Section 14]{Sato}). Since $\sigma$ is symmetric
  and not supported in a proper subspace of $\R^n$, the linear span of
  its support is all of $\R^n$, and \cite[Theorem 1.1(iii)]{Watanabe}
  then implies the first claim.

  Corollary 4.2 of \cite{Kanter} implies that every symmetric stable
  distribution on $\R^n$ is \emph{unimodal} in the sense defined in
  \cite{Kanter} and hence \emph{$n$-unimodal} in the sense defined in
  \cite{OlSa} (see discussion on p.\ 80 and p.\ 84 of
  \cite{Kanter}). The second claim then follows from \cite[Theorem
  6]{OlSa}.
\end{proof}

As in section \ref{S:compact-pd}, for a finite set $B \subseteq \R^n$
and $w \in \R^B$, we write $f_w(x) = \sum_{b \in B} w_b
F_p(x-b)$.
Recall that the RKHS $\mathcal{H}$ of the metric space $(\R^n, d_p)$
is the completion of the span of such functions $f_w$ with respect to
the norm given by
\[
\norm{f_w}_{\mathcal{H}}^2 = \sum_{a, b \in B} w_a
w_b F_p(a-b) = \int_{\R^n} \widehat{F_p}(x) \abs{\sum_{b
    \in B} w_b e^{2\pi i\inprod{x}{b}}}^2 \ dx
= \int_{\R^n} \frac{1}{\widehat{F_p}(x)} \abs{\widehat{f_w}(x)}^2 \ dx.
\]
Observe that the Fourier inversion theorem may be used here since
$\widehat{F_p}$ is the density of a random variable, hence integrable.

From here, standard arguments imply the following.

\begin{prop}
  \label{T:p-norm-RKHS}
  The RKHS of $(\R^n, d_p)$ is
  \[
  \mathcal{H} = \Set{ f \in L_2(\R^n)}{\int_{\R^n}
    \frac{1}{\widehat{F_p}(x)} \abs{\widehat{f}(x)}^2 \ dx < \infty},
  \]
  with norm given by
  \[
  \norm{f}_{\mathcal{H}}^2 = \int_{\R^n} \frac{1}{\widehat{F_p}(x)}
  \abs{\widehat{f}(x)}^2 \ dx.
  \]
  The Schwartz space $\mathcal{S}(\R^n)$ is contained in
  $\mathcal{H}$.
\end{prop}

The dual space of $\mathcal{H}$ is naturally identified with the space
of tempered distributions
\[
\Set{\varphi \in \mathcal{S}'(\R^n)}{ \widehat{\varphi} \in
  L_2\bigl(\widehat{F_p}(x)\ dx\bigr)}.
\]
Thus weightings for compact subsets of $(\R^n, d_p)$ can be identified
as tempered distributions satisfying a weak smoothness condition,
although we will not make use of this fact here.  Note that, since
$\widehat{F_p}$ is integrable, this space of distributions includes
all finite signed measures on $\R^n$, so that weight measures fit
gracefully into this perspective.

This concrete identification of the RKHS of $(\R^n, d_p)$, together
with Proposition \ref{T:nice-norms}, make it possible to use Fourier
analysis to prove a number of nice properties of magnitude in these
spaces, including the following fundamental fact.

\begin{prop}
  \label{T:p-norm-finite-mag}
  Let $A \subseteq (\R^n, d_p)$ be compact. Then 
  \[
  \frac{\vol A}{\Gamma \bigl(\frac{n}{p} + 1\bigr) \vol \mathcal{B}}
  \le \mg{A} < \infty.
  \]
\end{prop}

\begin{proof}
  By Proposition \ref{T:p-norm-RKHS}, $\mathcal{H}$ contains functions
  which are uniformly equal to $1$ on $A$, so the finiteness
  follows from Theorem \ref{T:inf}.

  For the lower bound, let $h$ be the potential function of $A$. By
  Theorem \ref{T:inf}, Proposition \ref{T:p-norm-RKHS}, Proposition
  \ref{T:nice-norms}(\ref{P:monotone}), and Plancherel's theorem,
  \[
  \mg{A} = \norm{h}_{\mathcal{H}}^2 \ge
  \frac{\bigl\Vert \widehat{h}\bigr\Vert_{2}^2}{\Gamma(\frac{n}{p}+1)
    \vol \mathcal{B}} =
  \frac{\norm{h}_{2}^2}{\Gamma(\frac{n}{p}+1)
    \vol \mathcal{B}} \ge \frac{\vol A}{\Gamma(\frac{n}{p}+1)
    \vol \mathcal{B}}.
  \qedhere
  \]
\end{proof}

The finiteness statement in Proposition \ref{T:p-norm-finite-mag} was
proved in Theorem 3.4.8 and Proposition 3.5.3 of \cite{MMS} for
$\ell_1^n$ and $\ell_2^n$, and in somewhat greater generality in
\cite[Theorem 4.3]{MeckPDM}. The lower bound was proved in
\cite[Theorem 3.5.6]{MMS} for $p=1$ and \cite[Theorem 4.5]{MeckPDM}
for the general case.\footnote{Theorem 4.5 in the published version of
  \cite{MeckPDM} is misstated in the case $p < 1$; see the current
  arXiv version for a correct statement.}  The proof here follows the
approach used in \cite{MeckMDC} for $\ell_2^n$ (see Proposition 5.6
and the remarks following Corollary 5.3 there).

We now consider the behavior of magnitude functions in $(\R^n, d_p)$.
We must be careful about a subtle notational issue when $p < 1$.
Recall that for a metric space $(A,d)$ and $t > 0$, we denote by $tA$
the metric space $(A, td)$, which in the present context is different
from the usual interpretation of $tA$.  Therefore we will introduce
the notation $t\cdot A = \Set{ta}{a \in A}$ for $A \subseteq
\R^n$.
Not that when $A \subseteq \R^n$ is equipped with the metric
$d_p(x,y) = \norm{x-y}^p$ associated to a $p$-norm, the metric space
$tA$ is isometric to the set $t^{1/p}\cdot A \subseteq \R^n$ equipped
with $d_p$.

The next result shows that magnitude knows about volume in all
finite-dimensional positive definite $p$-normed spaces.  This
generalizes \cite[Theorem 1]{BaCa} for Euclidean space $\ell_2^n$.

\begin{thm}
  \label{T:volume}
  If $A \subseteq (\R^n, d_p)$ is compact, then
  \[
  \lim_{t \to \infty} \frac{\mg{tA}}{t^{n/p}} =
  \lim_{t \to \infty} \frac{\mg{t\cdot A}}{t^{n}} = \frac{\vol
    A}{\Gamma\bigl(\frac{n}{p}+1\bigr) \vol \mathcal{B}}.
  \]
\end{thm}

\begin{proof}
  Proposition \ref{T:p-norm-finite-mag} implies that
  \[
  \mg{t\cdot A} \ge \frac{\vol (t \cdot
    A)}{\Gamma\bigl(\frac{n}{p}+1\bigr) \vol \mathcal{B}} =
  \frac{t^n \vol A}{\Gamma\bigl(\frac{n}{p}+1\bigr) \vol
    \mathcal{B}}
  \]
  for every $t > 0$.  Now suppose that $h \in \mathcal{H}$ satisfies
  $h \equiv 1$ on $A$, and let $h_t(x) = h(x/t)$. Then by Theorem
  \ref{T:inf} and Proposition \ref{T:p-norm-RKHS},
  \begin{equation}
    \label{E:mag-fn-upper-bound}
    \frac{\mg{t \cdot A}}{t^n} \le \int_{\R^n} \frac{1}{\widehat{F_p}(x)}
    \abs{\widehat{h_t} (x)}^2 \ dx = \int_{\R^n}
    \frac{1}{\widehat{F_p}(x/t)} \abs{\widehat{h}(x)}^2 \ dx.
  \end{equation}
  Proposition \ref{T:nice-norms}(\ref{P:monotone}), the
  monotone convergence theorem, and Plancherel's theorem imply that
  \[
  \lim_{t \to \infty} \int_{\R^n} \frac{1}{\widehat{F_p}(x/t)}
  \abs{\widehat{h}(x)}^2 \ dx =
  \frac{\norm{h}_{2}^2}{\Gamma\bigl(\frac{n}{p}+1\bigr) \vol
    \mathcal{B}}.
  \]
  By Theorem \ref{T:p-norm-RKHS}, there exist functions $h \in
  \mathcal{H}$ with $h \equiv 1$ on $A$ such that $\norm{h}_{2}^2$ is
  arbitrarily close to $\vol A$ (cf.\ the proof of \cite[Theorem
  1]{BaCa}), which completes the proof.
\end{proof}

The next theorem is the major known continuity result (as opposed to
mere \emph{semi}continuity) for magnitude.

\begin{thm}
  \label{T:star-continuity}
  Denote by $\mathcal{K}_n$ the class of nonempty compact subsets of
  $\R^n$, equipped with the Hausdorff metric $d_H$ induced by $d_p$,
  and suppose that $A \in \mathcal{K}_n$ is star-shaped with respect
  to some point in its interior.  Then magnitude, as a function
  $\mathcal{K}_n \to \R$, is continuous at $A$.
\end{thm}

\begin{proof}
  By Proposition \ref{T:semicont}, we only need to show that magnitude
  is upper semicontinuous at $A$. Letting $h$ be the potential
  function of $A$, \eqref{E:mag-fn-upper-bound} and Proposition
  \ref{T:p-norm-RKHS} imply that $\mg{t \cdot A} \le t^n \mg{A}$ for
  $t \ge 1$.  By translation-invariance, we may assume that $A$ is
  star-shaped about $0$ and $r^{1/p} \cdot \mathcal{B} \subseteq A$
  for some $r > 0$.  Now if $B \in \mathcal{K}_n$ and
  $d_H(A,B) < \eps$, then
  \[
  B \subseteq A + \eps^{1/p} \cdot \mathcal{B} \subseteq
  \left(1+ \left(\frac{\eps}{r}\right)^{1/p}\right) \cdot A,
  \]
  and so $\mg{B} \le \bigl(1+\bigl(\frac{\eps}{r}\bigr)^{1/p}\bigr)^n
  \mg{A}$. Thus magnitude is upper semicontinuous at $A$.
\end{proof}

The family of sets $A$ in Theorem \ref{T:star-continuity} is slightly
larger than what are sometimes called ``star bodies'', and of course
includes all convex bodies.  It is unknown, however, whether
magnitude is continuous when restricted to compact convex sets which
are not required to have nonempty interior.  

The final result in this section shows that, in positive definite
$p$-normed spaces, magnitude can be computed from potential functions
simply by integrating, as opposed to computing the (more complicated)
$\mathcal{H}$-norm.

\begin{thm}
  \label{T:magnitude-integral}
  Let $A \subseteq (\R^n, d_p)$ be compact, and suppose that the
  potential function $h \in \mathcal{H}$ of $A$ is integrable.  Then
  \[
  \mg{A} = \frac{1}{\Gamma\bigl(\frac{n}{p}+1\bigr) \vol \mathcal{B}}
  \int_{\R^n} h(x) \ dx.
  \]
\end{thm}

\begin{proof}
  Fix an even function $f \in \mathcal{S}(\R^n)$ with $f \equiv 1$ on
  some open neighborhood of the origin. Set $f_k(x) = f(x/k)$ and
  $\varphi_k = \widehat{f_k} / \widehat{F_p}$ for $k \in \N$.  Then
  $\varphi_k \in L_1 (\R^n)$ and
  \[
  \norm{\widehat{\varphi_k}}_\infty \le \norm{\varphi_k}_1 =
  \int_{\R^n} \frac{1}{\widehat{F_p}(x/k)} \abs{\widehat{f}(x)}\ dx \le
  \int_{\R^n} \frac{1}{\widehat{F_p}(x)} \abs{\widehat{f}(x)}\ dx <
  \infty
  \]
  by Proposition \ref{T:nice-norms}(\ref{P:monotone}).
  Furthermore, for every $x \in \R^n$,
  \[
  \widehat{\varphi_k}(x) = \int_{\R^n} e^{-2\pi i \inprod{x}{y/k}}
  \frac{\widehat{f}(y)}{\widehat{F_p}(y/k)} \ dy \xrightarrow{k\to
    \infty} \int_{\R^n} \frac{\widehat{f}(y)}{\widehat{F_p}(0)} \ dy =
  \frac{f(0)}{\Gamma\bigl(\frac{n}{p}+1\bigr) \vol \mathcal{B}} =
  \frac{1}{\Gamma\bigl(\frac{n}{p}+1\bigr) \vol \mathcal{B}}
  \]
  by the dominated convergence theorem.
  
  By the last part of Theorem \ref{T:inf}, for sufficiently large $k$,
  \[
  \mg{A} = \inprod{h}{f_k}_{\mathcal{H}} = \int_{\R^n} \widehat{h}(x)
  \varphi_k(x) \ dx = \int_{\R^n} h(x) \widehat{\varphi_k}(x) \ dx
  \]
  by Parseval's identity, and the claim now follows by the dominated
  convergence theorem.
\end{proof}

\subsection{Magnitude in Euclidean space}
\label{S:Euclidean}

Finally, we specialize the tools of section \ref{S:Fourier} to the
setting of Euclidean space $\ell_2^n$, where they become even more
powerful, allowing one to prove much more refined results about
continuity, asymptotics, and exact values of magnitude than in more
general normed spaces.

We will write simply $F(x) = e^{-\norm{x}_2}$, and let
$\ball{n}_R = \Set{x \in \R^n}{\norm{x}_2 \le R}$.  In this setting we
have the explicit formula
\begin{equation}
  \label{E:F-hat}
  \widehat{F}(x) = \frac{n! \omega_n}{(1 + 4 \pi^2
    \norm{x}_2^2)^{(n+1)/2}},
\end{equation}
where $\omega_n = \vol_n(\ball{n}_1)$ (see \cite[Theorem 1.14]{StWe}).
This implies that the RKHS $\mathcal{H}$ for $\ell_2^n$ is the
classical Sobolev space
\[
H^{(n+1)/2}(\R^n) = \Set{ f \in L_2(\R^n)}{\int_{\R^n} (1 + 4 \pi^2
  \norm{x}_2^2)^{(n+1)/2} \abs{\widehat{f}(x)}^2 \ dx < \infty},
\]
and that
$\norm{f}_{\mathcal{H}}^2 = \frac{1}{n! \omega_n}
\norm{f}_{H^{(n+1)/2}}^2$.

A first application of this observation is the following, proved for
$\ell_2^n$ in \cite[Corollary 5.5]{MeckMDC}.

\begin{thm}
  \label{T:mf-cont}
  If $A$ is a compact subset of $\ell_1^n$ or $\ell_2^n$, then the
  magnitude function $t \mapsto \mg{tA}$ is continuous on
  $(0,\infty)$.
\end{thm}

\begin{proof}[Sketch of proof]
  For $\ell_2^n$, using \eqref{E:F-hat} one can show that
  $\mg{tA} \ge \frac{1}{t} \mg{A}$ for $t \ge 1$, along the lines of
  \eqref{E:mag-fn-upper-bound}. For $\ell_1^n$, if we let
  $G(x) = e^{-\norm{x}_1} = \prod_{i=1}^n e^{-\abs{x_i}}$, then the
  $n=1$ case of \eqref{E:F-hat} implies that
  \[
  \widehat{G}(x) = \prod_{i=1}^n \frac{2}{1 + 4\pi^2 x_i^2},
  \]
  and a similar argument yields that $\mg{tA} \ge t^{-n} \mg{A}$ for
  $t \ge 1$.

  In either case, \eqref{E:mag-fn-upper-bound} shows that
  $\mg{tA} \le t^n \mg{A}$. Together, these estimates imply that the
  magnitude function of $A$ is continuous on $(0,\infty)$; see Theorem
  5.4 and Corollary 5.5 of \cite{MeckMDC}.
\end{proof}

The most significant consequence of \eqref{E:F-hat} is that when $n$
is odd, $1/\widehat{F}$ is the symbol of a differential operator on
$\R^n$.  In particular, when $f:\R^n \to \R$ is smooth,
\begin{equation}
  \label{E:H-norm-deriv}
  \norm{f}_{H^{(n+1)/2}}^2 = \int_{\R^n} f(x) \bigl[(I -
  \Delta)^{(n+1)/2} f\bigr](x) \ dx,
\end{equation}
where $I$ is the identity operator and $\Delta$ is the Laplacian on
$\R^n$.  This opens the door to using differential equations
techniques to study magnitude.  A first application is the proof of
the following result.

\begin{thm}[{\cite[Theorem 1]{BaCa}}]
  \label{T:shrink}
  If $A \subseteq \ell_2^n$ is compact, then
  $\lim_{t \to 0^+} \mg{tA} = 1$.
\end{thm}

\begin{proof}[Sketch of proof]
  By Proposition \ref{T:compact-monotone}, 
  it suffices to show that
  $\limsup_{R \to 0^+} \mg{\ball{n}_R} \le 1$; it further
  suffices, by embedding $\ell_2^n$ in $\ell_2^{n+1}$ if necessary, to
  assume that $n$ is odd.  For $0 < R < 1$ we can choose smooth
  functions $f_R$ such that
  \[
  f_R(x) = \begin{cases} 1 & \text{ if } \norm{x}_2 \le R, \\
    e^R e^{-\norm{x}_2} & \text{ if } \norm{x}_2 \ge \sqrt{R} \end{cases}
  \]
  and the derivatives of $f_R$ are sufficiently small for $R \le
  \norm{x}_2 \le \sqrt{R}$ that, using \eqref{E:H-norm-deriv},
  \[
  \norm{f_R}_{H^{(n+1)/2}}^2 = \norm{f_0}_{H^{(n+1)/2}}^2 + o(1)
  = n! \omega_n + o(1)
  \]
  when $R \to 0$; see the proof of \cite[Theorem 1]{BaCa}.  By Theorem
  \ref{T:inf}, this completes the proof.
\end{proof}
  
Together with Theorem \ref{T:mf-cont}, this shows that the magnitude
function of a compact $A \subseteq \ell_2^n$ is continuous on
$[0,\infty)$. Recall that this result is false for general metric
spaces $A$ of negative type \cite[Example 2.2.8]{MMS}, but it does
also hold for $A \subseteq \ell_1^n$
(Proposition~\ref{propn:ell1-mag-at-zero}). A monotone convergence
argument would prove the same result in a $p$-normed space if $\sup_x
\widehat{F_p}(x) / \widehat{F_p}(2x) < \infty.$

More significantly, we obtain the following conditions on the
potential function of a compact set $A \subseteq \ell_2^n$, which
provide the starting point for the only known approach for explicit
computation of magnitude for a convex body in $\ell_2^n$ when $n > 1$.
This result follows by considering the Euler--Lagrange equation of the
minimization problem in Theorem \ref{T:inf}, and applying elliptic
regularity.

\begin{thm}[Proposition 5.7 and Corollary 5.8 of \cite{MeckMDC}]
  \label{T:pde}
  Suppose that $n$ is odd and $A \subseteq \ell_2^n$ is compact. Then
  the potential function $h$ of $A$ is $C^\infty$ on $\R^n \setminus
  A$, and satisfies
  \begin{equation}
    \label{E:E-L}
    (I - \Delta)^{(n+1)/2} h (x) = 0
  \end{equation}
  on $\R^n \setminus A$.
\end{thm}



To indicate the usefulness of this observation, we show how Theorem
\ref{T:pde} can be used to quickly compute the magnitude of an
interval in $[a, b] \subseteq \R$.  By Theorem \ref{T:pde}, the
potential function $h$ satisfies $h - h'' = 0$ outside $[a,b]$.
The boundary conditions $h(x) = 1$ for $a\le x \le b$ and
$h(x) \to 0$ when $\abs{x} \to \infty$ (since $h \in H^1(\R^n)$) imply
that
\[
h(x) = \begin{cases} e^{x-a} & \text{ if } x < a, \\
  1 & \text{ if } a \le x \le b, \\
  e^{b-x} & \text{ if } x > b. \end{cases}
\]
Then by Theorem \ref{T:magnitude-integral},
\[
\mg{[a,b]} = \frac{1}{2} \int_\R h(x) \ dx = 1 + \frac{b-a}{2},
\]
in agreement with \eqref{E:interval-magnitude}.  A more involved, but
still elementary computation yields another proof of Corollary
\ref{T:R-mag-gaps}.

For higher dimensions, Barcel\'o and Carbery \cite{BaCa} analyzed the
minimization problem in more depth, and proved the following result
using standard techniques of the theory of partial differential
equations.

\begin{prop}[See Proposition 2 and Lemma 4 of \cite{BaCa}]
  \label{T:pde-uniqueness}
  Suppose that $n$ and $m$ are positive integers, and
  $A \subseteq \R^n$ is a convex body.
  \begin{enumerate}
  \item There is a unique function $f \in H^m(\R^n)$ such that
    \[
    (I - \Delta)^m f(x) = 0 \text{ on } \R^n \setminus A
    \]
    weakly and $f \equiv 1$ on $A$.
  \item If $\partial A$ is piecewise $C^1$ and $f \in H^m(\R^n)$, then
    all derivatives of $f$ up to order $m-1$ vanish on $\partial A$
    (in the sense of traces of Sobolev functions).
  \end{enumerate}
\end{prop}

Together with Theorems~\ref{T:magnitude-integral} and~\ref{T:pde},
Proposition~\ref{T:pde-uniqueness} reduces the computation of
magnitudes (in many cases) to the solution of a PDE boundary value
problem.  In general, of course, solving a PDE boundary value problem
is no simple matter.  But in the case that $A = \ball{n}_R$ is a
Euclidean ball, rotational symmetry reduces the partial differential
equation to an ordinary differential equation on $[R,\infty)$, albeit
of high degree. Barcel\'o and Carbery gave an algorithm for solving
the resulting ODE boundary value problem, and hence determining the
potential function $h$ of $\ball{n}_R$, for every odd dimension $n$
and radius $R > 0$. From there, Theorem \ref{T:magnitude-integral} can
be used to compute the magnitude $\ball{n}_R$.  (In \cite{BaCa} the
magnitude was found by computing $\norm{h}_{H^{(n+1)/2}}^2$ using
\eqref{E:H-norm-deriv}, since Theorem \ref{T:magnitude-integral} had
not yet been proved; Theorem \ref{T:magnitude-integral} makes the
computation much simpler.)  This approach yields the following.

\begin{thm}[Theorems 2, 3, and 4 of \cite{BaCa}]
  \label{T:l2-balls}
  For every $R > 0$,
  \[
  \mg{\ball{3}_R} = 1 + 2R + R^2 + \frac{1}{6} R^3
  \]
  and
  \[
  \mg{\ball{5}_R} = \frac{24 + 72R^2 + 35 R^3 + 9R^4 +
    R^5}{8(R+3)} + \frac{1}{120} R^5.
  \]
  In general, when $n$ is odd, the magnitude
  $\mg{\ball{n}_R}$ is a rational function of $R > 0$ with
  rational coefficients.
\end{thm}

Barcel\'o and Carbery also give an explicit formula for
$\mg{\ball{7}_R}$.  We recall that they also determined the
asymptotics of $\mg{\ball{n}_R}$ when $R \to 0$ and $R \to \infty$ in
\cite[Theorem 1]{BaCa}, stated above in Theorems \ref{T:shrink} and
\ref{T:volume}.  To date, odd-dimensional balls are the only convex
bodies in Euclidean space whose exact magnitudes are known.

It was previously conjectured in \cite{LeWi} that for a compact convex
set $A \subseteq \ell_2^n$,
\begin{equation}
  \label{E:convex-conj}
  \mg{A} = \sum_{i=0}^n \frac{V_i(A)}{i! \omega_i},
\end{equation}
where $V_0, \dots, V_n$ denotes the classical intrinsic volumes.
Theorem \ref{T:l2-balls} implies that \eqref{E:convex-conj} holds for
balls in $\ell_2^3$, but is false in dimensions $n \ge 5$.

To put this conjecture in context, observe that \eqref{E:convex-conj}
is a Euclidean version of Conjecture~\ref{C:ell1-convex} in
$\ell_1^n$.  Note that $2^i = i! \omega_i'$ , where $\omega_i'$
denotes the volume of the unit ball in $\ell_1^i$, tightening the
analogy between \eqref{E:convex-conj} and
Conjecture~\ref{C:ell1-convex}. At the time that \eqref{E:convex-conj}
was proposed, it was known to hold for $n=1$, and was supported by
numerical computations in $n=2$
\cite{Willerton-heuristic}. Furthermore, some cases of Conjecture
\ref{C:ell1-convex} in $\ell_1^n$ (contained in Theorem
\ref{thm:ell1-cvx}) were known to be true.

Several interesting questions remain open, most obviously whether
\eqref{E:convex-conj} holds for $n \le 4$.  Noting that
\eqref{E:convex-conj} is equivalent to
\begin{equation}
  \label{E:convex-conj-mf}
    \mg{tA} = \sum_{i=0}^n \frac{V_i(A)}{i! \omega_i} t^i,
\end{equation}
Proposition \ref{T:p-norm-finite-mag} and Theorem \ref{T:volume} say
that \eqref{E:convex-conj-mf} is true to top order for sets of
positive volume when $t \to \infty$, and Theorem \ref{T:shrink} shows
that \eqref{E:convex-conj} predicts the correct behavior when $t \to
0$.  One could ask whether \eqref{E:convex-conj-mf} is approximately
true in some sharper asymptotic senses.  Note that Theorem
\ref{T:homog-manifold} is a Riemannian analogue of a weak asymptotic
version of \eqref{E:convex-conj-mf}.  Barcel\'o and Carbery also raise
the question of whether \eqref{E:convex-conj} holds if magnitude is
replaced by a suitable modification which coincides with magnitude in
$\ell_2^n$ for $n \le 3$.  We mention another related question in
section \ref{S:open-problems}.

The final major consequence of the concrete identification of
$\mathcal{H}$ for Euclidean space is the realization that magnitude
and maximum diversity, in the setting of $\ell_2^n$, are actually
classical notions of capacity, well-known in potential theory.  The
formal similarity between magnitude and maximum diversity on the one
hand, and capacity on the other, is clear from the definitions (cf.\
section 1.1 of \cite{BaCa}).  But in $\ell_2^n$, magnitude and maximum
diversity almost precisely reproduce classically studied forms of
capacity.

Specifically, \eqref{E:F-hat} and \cite[Theorem 2.2.7]{AdHe} imply
that for a compact set $A \subseteq \ell_2^n$,
$\md{A} = \frac{1}{n! \omega_n} C_{(n+1)/2}(A)$, where
\[
C_\alpha(A) := \inf \Set{\norm{f}_{H^{\alpha}}^2}{f
    \in \mathcal{S}(\R^n),\ f \ge 1 \text{ on } A},
\]
is the \dfn{Bessel capacity} of $A$ of order $\alpha$. An alternative
notion of capacity, which naturally arises in the study of
removability of singularities (see \cite[Section 2.7]{AdHe}), is
\[
N_\alpha(A) := \inf \Set{\norm{f}_{H^{\alpha}}^2}{f \in
    \mathcal{S}(\R^n),\ f \ge 1 \text{ on a neighborhood of } A}.
\]
By Theorem \ref{T:inf},
$\mg{A} \le \frac{1}{n! \omega_n} N_{(n+1)/2}(A)$.  In fact one would
expect $\mg{A} = \frac{1}{n! \omega_n} N_{(n+1)/2}(A)$; this appears
not to be the case for arbitrary compact $A$, but happily the
comparison we have will be enough for our purpose.

Before moving on, we pause to observe that, although we have just seen
that magnitude and its cousin maximum diversity fit into classical
families of capacities, they both just fail to fit into the parameter
range which is of relevance for classical applications.  As alluded to
above, capacities are frequently used to quantify ``exceptional''
sets; sets of capacity $0$ are a frequent substitute for sets of
measure $0$ when studying singularities.  However, $C_\alpha(A)$ and
$N_\alpha(A)$ are bounded below by positive constants whenever
$\alpha > n/2$. So from the point of view of classical potential
theory, magnitude and maximum diversity are rather pathological.
Nevertheless, the following result from potential theory, whose main
classical application is to show that $C_\alpha(A) = 0$ if and only if
$N_\alpha(A) = 0$, also applies in our setting.

\begin{prop}[{\cite[Theorem 3.3.4]{AdHe}}]
  \label{T:cap-equivalence}
  For each $n$ and each $\alpha > 0$ there is a constant
  $\kappa_{n, \alpha} \ge 1$ such that, for every compact set
  $A \subseteq \ell_2^n$,
  \[
  C_{\alpha}(A) \le N_{\alpha}(A) \le \kappa_{n,\alpha} C_{\alpha}(A).
  \]
\end{prop}

\begin{cor}[{\cite[Corollary 6.2]{MeckMDC}}]
  \label{T:mag-md-equivalence}
  For each $n$ there is a constant $\kappa_n \ge 1$ such that, for
  every compact set $A \subseteq \ell_2^n$,
  \[
  \md{A} \le \mg{A} \le \kappa_n \md{A}.
  \]
\end{cor}

The significance of Corollary \ref{T:mag-md-equivalence} is that,
although maximum diversity is no easier to compute explicitly than
magnitude, in some ways its rough behavior is easier to analyze.  For
example, it is natural to conjecture that the magnitude function
$t \mapsto \mg{tA}$ is increasing for a compact space $A$ of negative
type.  It is unknown whether this is true. On the other hand, it is
obvious that $t \mapsto \mg{tA}_+$ is increasing, and Corollary
\ref{T:mag-md-equivalence} therefore implies that the magnitude
function of a compact set $A \subseteq \ell_2^n$ is at least bounded
above and below by constant multiples of an increasing function.

A more substantial consequence of Corollary \ref{T:mag-md-equivalence}
is the following result, which, like Theorem \ref{T:volume}, shows
that the category-theoretically inspired notion of magnitude turns out
to encode quantities of fundamental importance in geometry.

\begin{thm}[{\cite[Corollary 7.4]{MeckMDC}}]
  \label{T:mdim}
  If $A \subseteq \ell_2^n$ is compact, then
  \[
  \lim_{t \to \infty} \frac{\log \mg{tA}}{\log t} = \boxdim A.
  \]
\end{thm}

Theorem \ref{T:mdim}, which should be interpreted in the same sense as
Proposition \ref{T:ddim}, follows immediately from Proposition
\ref{T:ddim} and Corollary \ref{T:mag-md-equivalence}.  Another
interesting aspect of this result is that, as noted above, classically
Proposition \ref{T:cap-equivalence} is of interest primarily for sets
of capacity $0$, or more generally for small sets; here it is instead
applied to large sets.

\section{Open problems}
\label{S:open-problems}

There are many interesting open problems about magnitude.  These
include extending partial results discussed above, as well as some
quite basic questions about the behavior of magnitude.  We mention
several of them below.

\begin{enumerate}
\item Does every compact positive definite space (or space of negative
  type) have finite magnitude?

  Proposition \ref{T:p-norm-finite-mag} implies that every compact
  subset of a finite dimensional positive definite normed (or
  $p$-normed) space has finite magnitude, so that the obvious place to
  look for a counterexample is in infinite dimensions. Essentially the
  only infinite-dimensional spaces whose magnitudes are known are
  boxes in $\ell_1$, which just miss being a counterexample:
  \[
  \mg{\prod_{i=1}^\infty [0,r_i]} = \prod_{i=1}^\infty \left(1 +
    \frac{r_i}{2}\right).
  \]
  The condition $\norm{r}_1 < \infty$, which both guarantees that this
  infinite-dimensional box lies in $\ell_1$ and is compact, is also
  equivalent to the finiteness of the product on the right-hand side.

\item Is magnitude continuous on the class of compact sets in a
  positive definite normed (or $p$-normed) space? What if we assume
  the space is finite-dimensional, or we restrict to geodesic sets, or
  convex sets?

  Recall that magnitude is continuous on convex bodies in a
  finite-dimensional positive definite $p$-normed space (Theorem
  \ref{T:star-continuity}), but is not continuous on the class of
  compact spaces of negative type (Examples 2.2.8 and 2.4.9 of
  \cite{MMS}).

\item Is Conjecture \ref{C:ell1-convex} true? Is it at least true for
  compact convex sets $A \subseteq \ell_1^n$?

  In light of Theorem \ref{thm:ell1-cvx}, Conjecture
  \ref{C:ell1-convex} is equivalent to the continuity of magnitude on
  compact, geodesic (i.e., $\ell_1$-convex) sets in
  $\ell_1^n$. Similarly, if magnitude is continuous on compact convex
  sets in $\ell_1^n$, then Theorem \ref{thm:ell1-cvx} would imply
  that Conjecture \ref{C:ell1-convex} holds for compact convex sets.

\item Does the magnitude function of a convex body $A \subseteq
  \ell_2^n$ determine its intrinsic volumes?  What about a homogeneous
  compact Riemannian manifold?

\item Does it hold that
  \[
  \mg{t(A \cup B)} + \mg{t(A \cap B)} - \mg{tA} - \mg{tB} \to 0
  \]
  as $t \to \infty$ for compact, convex sets $A,B \subseteq \ell_2^n$
  (or in more general normed spaces) such that $A \cup B$ is convex?

  For convex bodies in $\ell_1^n$, the left-hand side of the above is
  $0$ for every $t$, as a consequence of Theorem \ref{thm:ell1-cvx};
  the same would be true in $\ell_2^n$ if \eqref{E:convex-conj} were
  true.

\item Does Theorem \ref{T:mdim} hold for arbitrary compact spaces of
  negative type?

  Theorem \ref{T:mdim} applies to compact subsets of $\ell_2^n$.  As
  mentioned earlier, Proposition 7.5 of \cite{MeckMDC} shows that the
  conclusion of Theorem \ref{T:mdim} also holds for compact
  homogeneous metric spaces.  In addition, Theorem \ref{T:volume}
  implies that the conclusion of Theorem \ref{T:mdim} holds for
  compact subsets of positive $n$-dimensional volume in an
  $n$-dimensional positive definite $p$-normed space; hence it holds,
  for example, for compact convex sets in any positive definite
  $p$-normed space.
  
\end{enumerate}

\bibliographystyle{plain}
\bibliography{magnitude-survey}

\end{document}